\documentclass{amsart}
\usepackage{amssymb,amsmath}
\usepackage{graphicx}
\usepackage{enumerate}
\usepackage{ifthen}

\DeclareMathOperator{\tr}{tr}

\DeclareMathOperator{\diag}{diag}

\def\R{{\mathbb{R}}}

\def\P{{\mathcal{P}}}

\def\L{{\mathcal{L}}}
\def\dt{{\frac{d}{dt}}}

\def\li{{\lambda_i}}
\def\lj{{\lambda_j}}
\def\theta{{\vartheta}}
\def\phi{{\varphi}}
\def\epsilon{{\varepsilon}}
\long\def\umbruch{{\displaybreak[1]}}

\def\eg{e.\,g.\ }
\def\ie{i.\,e.\ }

\def\fracp#1#2{{\frac{\partial #1}{\partial #2}}}
\def\lx{\lambda_1}
\def\ly{\lambda_2}
\newcommand{\A}[1]{\ifthenelse{#1 = 2}{\lvert A\rvert^{#1}}{\tr A^{#1}}}

\mathchardef\ordinarycolon\mathcode`\:
  \mathcode`\:=\string"8000
  \begingroup \catcode`\:=\active
    \gdef:{\mathrel{\mathop\ordinarycolon}}
  \endgroup

\newtheorem{theorem}{Theorem}[section]
\newtheorem{lemma}[theorem]{Lemma}

\newtheorem{corollary}[theorem]{Corollary}

\theoremstyle{definition}
\newtheorem{definition}[theorem]{Definition}

\newtheorem{remark}[theorem]{Remark}

\numberwithin{equation}{section}
\begin{document}
\title[Maximum-Principle Functions]{On Maximum-Principle Functions for Flows by Powers of the Gauss Curvature}

\author{Martin Franzen}
\address{Martin Franzen, Universit\"at Konstanz,
   Universit\"atsstrasse 10, 78464 Konstanz, Germany}
\curraddr{}
\def\ukaddress{@uni-konstanz.de}
\email{Martin.Franzen\ukaddress}
\thanks{We would like to thank Oliver Schn\"urer and Martin Raum for discussions and support.\\ The author is a member of the DFG priority program SPP 1489.}

\subjclass[2000]{53C44}

\date{December 18, 2013.}


\keywords{}

\begin{abstract}
  We consider flows with normal velocities equal to powers strictly larger than one of the Gauss curvature. 
  Under such flows closed strictly convex surfaces converge to points.
  In his work on $|A|^2$, Schn\"urer proposes criteria for selecting 
  quantities that are suitable for proving convergence to a round point.
  Such monotone quantities exist for many normal velocities,
  including the Gauss curvature, some powers larger than one of the mean curvature, and
  some powers larger than one of the norm of the second fundamental form.
  In this paper, we show that no such quantity exists for any powers larger than one of the Gauss curvature.
\end{abstract}

\maketitle


\section{Introduction}\label{sec:introduction}

We consider closed strictly convex surfaces $M_t$ in $\R^3$ that contract with normal velocities equal to the positive powers of the Gauss curvature,
\begin{equation}\label{flow equation K}
  \frac{d}{dt} X = -K^{\sigma}\nu.
\end{equation}
For all $\sigma > 0$, this is a parabolic flow equation. 
We have a solution on a maximal time interval $[0,T)$, $0 < T < \infty$. 
Chow \cite{bc:deforming} proves that the surfaces converge to a point as $t \rightarrow T$.

For $\sigma = 1$, this is the Gauss curvature flow. 
It was introduced by Firey as a model for the shape of wearing stones on beaches \cite{wf:shapes}.
Firey conjectured that, after appropriate rescaling, the surfaces converge to spheres.
This is also referred to as convergence to a ``round point''.
The conjecture was confirmed by Andrews in \cite{ba:gauss}.
Andrews and Chen \cite{ac:surfaces} extended this result to all powers $\frac{1}{2} \leq \sigma \leq 1$.
The crucial step in their proof, Theorem 2.2, is to show that the quantity
\begin{equation}\label{monotone quantity K}
  \max_{M_t} \frac{\left(\lx-\ly\right)^2}{\lx^2\ly^2} \cdot K^{2\sigma}
\end{equation}
is non-increasing in time. $\lx$, $\ly$ denote the principal curvatures of the surfaces $M_t$.

We give a more detailed introduction to the standard notation in Section \ref{sec notation}.

For many other normal velocities $F$ monotone quantities $w$ like \eqref{monotone quantity K} are known,
which are monotone during the corresponding flows and vanish precisely for spheres. 
For $F=|A|^2$ Schn\"urer \cite{os:surfacesA2} obtains $$w = \frac{\left(\lx-\ly\right)^2}{\lx\ly}\cdot H$$ and 
for $F=H^\sigma,\; 1 \leq \sigma \leq \sigma_{\ast},\; \sigma_{\ast} \approx 5.17$, Schulze and Schn\"urer \cite{fs:convexity} use
$$w=\frac{\left(\lx-\ly\right)^2}{\lx^2\ly^2}\cdot H^{2\sigma}.$$
For $F=\tr A^\sigma,\; 1 \leq \sigma < \infty$, Andrews and Chen \cite{ac:surfaces} get $$w=\frac{\left(\lx-\ly\right)^2}{\lx^2\ly^2}\cdot\left(\tr A^\sigma\right)^2.$$
This is the first example where flows of an arbitrarily high degree of homogeneity converge surfaces to round points.

In \cite{os:surfacesA2}, Schn\"urer proposes criteria for selecting monotone quantities like \eqref{monotone quantity K}.
To the author's knowledge, to date, all known quantities which fulfill these criteria can be used to prove convergence to a round point.
This is why we work with these criteria as a definition in this paper.
Our question is whether such monotone quantities exist for equation \eqref{flow equation K} if $\sigma > 1$.
Their monotonicity is proven using the maximum-principle 
so we name these quantities \textit{maximum-principle functions}.

\begin{definition}\label{def mpf}
  \textbf{(Maximum-principle function)}
  Let $w$ be a symmetric rational function of the principal curvatures,
  \begin{align}
    w = \frac{p\left(\lx,\ly\right)}{q\left(\lx,\ly\right)}.
  \end{align}
  Here, $p \not\equiv 0$ and $q \not \equiv 0$ are homogeneous polynomials. 
  $w$ is called a \textit{maximum-principle function} for a normal velcocity $F$, 
  if 
  \begin{enumerate}
    \item\label{I}
      \begin{enumerate}
        \item $p(\lx,\,\ly)\ge0,\,q(\lx,\,\ly)>0$ for all $0<\lx,\,\ly$,
        \item $p(\lx,\,\ly)=0$ for $\lx=\ly>0$.
      \end{enumerate}
    \item\label{II}
      $\deg p>\deg q$.
    \item\label{III}
      $\fracp{w(\rho,1)}{\rho}<0$ for all $0<\rho<1$ and 
      $\fracp{w(\rho,1)}{\rho}>0$ for all $\rho>1$.
    \item\label{IV}
      $L\left(w\right) := \frac{d}{dt} w-F^{ij}w_{;\,ij} \leq 0$ for all $0<\lx,\,\ly.$ \\
      We achieve this by assuming
      \begin{enumerate}
        \item terms without derivatives of $\left(h_{ij}\right)$ are nonpositive, and
        \item terms involving derivatives of $\left(h_{ij}\right)$ at a critical point of $w$, 
        \textit{i.e.} $w_{;i} =0$ for all $i=1,\,2$, are nonpositive.
     \end{enumerate}
   \end{enumerate}
\end{definition}
As in \cite{os:surfacesA2,os:surfaces}, we motivate conditions \eqref{I} to \eqref{IV}.
For all flow equations considered, spheres contract to round points.
So we can only find monotone quantities if $\deg p \leq \deg q$ or $p\left(\lambda,\lambda\right) = 0$ for all $\lambda>0$.

If $\deg p < \deg q$, we obtain that $w$ is non-increasing on any self-similarly contracting surface.
So this does not imply convergence to a round point.

Condition \eqref{III} ensures that the quantity decreases, if the ratio of the principal curvatures $\lx/\ly$ approaches one.

By condition \eqref{IV} we check that we can apply the maximum-principle to prove monotonicity. \\

The linear operator $L\left(w\right)$, which corresponds to the general flow equation $$\frac{d}{dt} X = -F \nu,$$ fulfills an identity of the form 
   \begin{align*}
     L\left(w\right) = C_w\left(\lx,\ly\right) + G_w\left(\lx,\ly\right) h_{11;1}^2+ G_w\left(\ly,\lx\right)h_{22;2}^2,
   \end{align*}
at a critical point of $w$. This is Lemma \ref{lem evolution equation w}. 
We name the rational function $C_w\left(\lx,\ly\right)$ the \textit{constant terms} and 
the rational function $G_w\left(\lx,\ly\right)$ the \textit{gradient terms} of the evolution equation. 
To fulfill condition \eqref{IV} the constant terms $C_w\left(\lx,\ly\right)$ and the gradient terms $G_w\left(\lx,\ly\right)$ 
simultaneously have to be nonpositive for all $0 < \lx,\,\ly$. 
Here, we obtain a contradiction for $F=K^{\sigma}$ if $\sigma>1$.

   \text{\quad\, Our main theorem is}
\begin{theorem}\label{main theorem}
  For  a family of smooth closed strictly convex surfaces $M_t$ in $\R^3$ flowing according to
  \begin{equation*}
    \frac{d}{dt} X = -K^{\sigma}\nu, \quad \sigma > 1,
  \end{equation*}
  there exist no \textit{maximum-principle functions}.
\end{theorem}
   
   Despite this fact, it remains an open question whether for any powers $\sigma > 1$, closed strictly convex surfaces converge to round points.
   Due to Andrews we already know that this does not necessarily happen for all powers $\frac{1}{4} \leq \sigma \leq \frac{1}{2}$.
   For $\sigma = \frac{1}{4}$, they converge to ellipsoids \cite{ba:contraction}. For all powers $\frac{1}{4} < \sigma \leq \frac{1}{2}$, 
   surfaces contract homothetically in the limit \cite{ba:motion}. \\
   
   In Section \ref{sec proof strategy} we explain the proof strategy. In Section \ref{sec evolution equations} and \ref{sec dehomogenized polynomials} we outline the proof of our main Theorem \ref{main theorem}.
   
\section{Notation}\label{sec notation}
  
For this paper, we adopt the chapter on standard notation from \cite{os:surfacesA2}.\\

The linear operator $L$ corresponding to the general flow equation 
\begin{equation}\label{eq general flow equation}
  \frac{d}{dt} X = -F\nu
\end{equation}
is defined by
\begin{equation}\label{def linear operator L}
  L\left(w\right) := \frac{d}{dt} w - F^{ij} w_{;ij}.
\end{equation}

We use $X=X(x,\,t)$ to denote the embedding vector of a manifold $M_t$ into $\R^3$ and 
$\dt X=\dot{X}$ for its total time derivative. 
It is convenient to identify $M_t$ and its embedding in $\R^3$.
The normal velocity $F$ is a homogeneous symmetric function of the principal curvatures.
We choose $\nu$ to be the outer unit normal vector to $M_t$. 
The embedding induces a metric $g_{ij} := \langle X_{,i},\, X_{,j} \rangle$ and
the second fundamental form $h_{ij} := -\langle X_{,ij},\,\nu \rangle$ for all $i,\,j = 1,\,2$. 
We write indices preceded by commas to indicate differentiation with respect to space components, 
\eg $X_{,k} = \frac{\partial X}{\partial x_k}$ for all $k=1,\,2$.

We use the Einstein summation notation. 
When an index variable appears twice in a single term it implies summation of that term over all the values of the index.

Indices are raised and lowered with respect to the metric
or its inverse $\left(g^{ij}\right)$, 
\eg $h_{ij} h^{ij} = h_{ij} g^{ik} h_{kl} g^{lj} = h^k_j h^j_k$.

The principal curvatures $\lx,\,\ly$ are the eigenvalues of the second fundamental
form $\left(h_{ij}\right)$ with respect to the induced metric $\left(g_{ij}\right)$. A surface is called strictly convex,
if all principal curvatures are strictly positive.
We will assume this throughout the paper.
Therefore, we may define the inverse of the second fundamental
form denoted by $\left(\tilde h^{ij}\right)$.

Symmetric functions of the principal
curvatures are well-defined, we will use the mean curvature
$H=g^{ij} h_{ij} = \lx+\ly$, the square of the norm of the second fundamental form
$\A2= h^{ij} h_{ij} = \lx^2+\ly^2$, the trace of powers of the second fundamental form 
$\tr A^{\sigma} = \tr \left(h^i_j\right)^{\sigma} = \lx^{\sigma} +\ly^{\sigma}$, and the Gauss curvature
$K= \frac{\det h_{ij}}{\det g_{ij}} = \lx\ly$. We write indices preceded by semi-colons 
to indicate covariant differentiation with respect to the induced metric,
e.\,g.\ $h_{ij;\,k} = h_{ij,k} - \Gamma^l_{ik} h_{lj} - \Gamma^l_{jk} h_{il}$, 
where $\Gamma^k_{ij} = \frac{1}{2} g^{kl} \left(g_{il,j} + g_{jl,i} - g_{ij,l}\right)$.
It is often convenient to choose normal coordinates, \ie coordinate systems such
that at a point the metric tensor equals the Kronecker delta, $g_{ij}=\delta_{ij}$,
and $(h_{ij})$ is diagonal, $(h_{ij})=\diag(\lx,\,\ly)$. 
Whenever we use this notation, we will also assume that we have 
fixed such a coordinate system. We will only use Euclidean coordinate
systems for $\R^3$ so that the indices of $h_{ij;\,k}$ commute according to
the Codazzi-Mainardi equations.

A normal velocity $F$ can be considered as a function of $(\lx,\,\ly)$
or $(h_{ij},\,g_{ij})$. We set $F^{ij}=\fracp{F}{h_{ij}}$,
$F^{ij,\,kl}=\fracp{^2F}{h_{ij}\partial h_{kl}}$. 
Note that in coordinate
systems with diagonal $h_{ij}$ and $g_{ij}=\delta_{ij}$ as mentioned
above, $F^{ij}$ is diagonal. 
For $F=K^\sigma$, we have $F^{ij}=\sigma K^{\sigma} \tilde{h}^{ij}$.

\section{Proof strategy}\label{sec proof strategy}

To prove our main Thoreom \ref{main theorem}, we use an elementary fact about polynomials in one variable. If a polynomial in one variable $\rho$, which is not constantly zero, is nonpositive for all $\rho > 0$, then the coefficient of its leading term has to be negative. As mentioned before, we focus on condition \eqref{IV} in the Defintion \ref{def mpf} of the \textit{maximum-principle functions}. We use an indirect proof and assume the existence of  a \textit{maximum-principle function}. We calculate the constant terms $C_w\left(\lx,\ly\right)$ and the gradient terms $G_w\left(\lx,\ly\right)$ for general homogeneous symmetric polynomials $p\left(\lx,\ly\right)$ and $q\left(\lx,\ly\right)$. We state them in the algebraic basis $\lbrace H,K \rbrace$, where $H$ is the mean curvature and $K$ is the Gauss curvature, \ie $p(H,K) :=\,\sum_{i=0}^{\lfloor g/2\rfloor} c_{i+1}H^{g-2i}K^i$. Using the identity $q^2 L\left(p/q\right) = q\,L\left(p\right) - p\,L\left(q\right)$ at a critical point of $p/q$, where we also choose normal coordinates, we easily see that $C_w\left(\lx,\ly\right)$ and $G_w\left(\lx,\ly\right)$ differ from a polynomial only by a nonnegative factor. Dividing by this nonnegative factor, we transform $C_w\left(\lx,\ly\right)$ and $G_w\left(\lx,\ly\right)$ into polynomial versions of the constant terms and the gradient terms. The next step is to dehomogenize the polynomial version of  $C_w\left(\lx,\ly\right)$ and $G_w\left(\lx,\ly\right)$ by setting $\lx=\rho,\,\ly=1$ and vice versa. Since the polynomial version of the constant terms is symmetric and the polynomial version of the gradient terms is asymmetric, we obtain only three instead of four polynomials in one variable $C\left(\rho\right)$, $G_1\left(\rho\right)$ and $G_2\left(\rho\right)$. 
Due to the property of a \textit{maximum-principle function}, all three polynomials in $\rho$ have to be nonpositive for all $\rho \geq 0$.

Now we calculate the leading terms of all three polynomials in one variable. Here, for technical reasons we have to distinguish nine cases. However, each case results in a contradiction. As it turns out, the coefficients of the leading terms of $C\left(\rho\right)$, $G_1\left(\rho\right)$ and $G_2\left(\rho\right)$ never can be simultaneously negative. This concludes the proof of our main Theorem \ref{main theorem}. 

\section{Evolution equations, constant terms and gradient terms}\label{sec evolution equations}
  
\subsection{Evolution equations}

In the first part of Section \ref{sec evolution equations}, we do some preliminary work.
We calculate the covariant derivatives of the mean curvature $H$ and the Gauss curvature $K$. 
  Furthermore, we present the evolution equations, corresponding to the general flow equation \eqref{eq general flow equation} of the following geometric quantities
  \begin{itemize}
    \item induced metric $g_{ij}$,
    \item inverse of the induced metric $g^{ij}$,
    \item second fundamental form $h_{ij}$,
    \item mean curvature $H$,
    \item Gauss curvature $K$,
    \item and general function $w\left(H,K\right)$ depending on $H$ and $K$.
  \end{itemize}
  
\begin{lemma}
  The covariant derivative of the mean curvature $H$ is given by
  \begin{equation}
    \label{H diff}
    H_{;k} = g^{ij} h_{ij;k}.
  \end{equation}
\end{lemma}
\begin{proof}
  Direct calculations yield
  \begin{align*}
    H_{;k}=&\, \left(g^{ij} h_{ij}\right)_{;k} \umbruch\\
             =&\, \underbrace{ \left( g^{ij} \right)_{;k} }_{=\, 0} h_{ij} + g^{ij} h_{ij;k} \umbruch\\
             =&\, g^{ij} h_{ij;k}.\qedhere
  \end{align*}
\end{proof}

\begin{lemma}
  The covariant derivative of the Gauss curvature $K$ is given by
  \begin{equation}
    \label{K diff}
    K_{;k}=\,K \tilde{h}^{ij} h_{ij;k}.
  \end{equation}
\end{lemma}
\begin{proof}
  Direct calculations yield
  \begin{align*}
    K_{;k}=&\,\left(\frac{\det h_{ij} }{\det g_{ij} }\right)_{;k} \umbruch\\
    =&\, \frac{1}{\left(\det g_{ij}\right)^2} \left( \left(\det h_{ij}\right)_{;k}\left(\det g_{ij}\right) - \left(\det h_{ij}\right) \left(\det g_{ij}\right)_{;k} \right) \umbruch\\
    =&\, \frac{1}{\left(\det g_{ij}\right)^2} \Bigg( \left(\frac{\partial}{\partial h_{ij}} \det h_{ij}\right) \left(h_{ij;k}\right) \left(\det g_{ij} \right) 
      -\left(\det h_{ij}\right)\left(\frac{\partial}{\partial g_{ij}} \det g_{ij}\right) \underbrace{\left(g_{ij;k}\right)}_{=\,0} \Bigg) \umbruch\\
    =&\, \frac{1}{\det g_{ij}} \left(\det h_{ij}\right) \tilde{h}^{ji} h_{ij;k} \umbruch\\
    =&\, K \tilde{h}^{ij} h_{ij;k}.\qedhere
  \end{align*}
\end{proof}
\begin{lemma}
  The metric $g_{ij}$ evolves according to 
  \begin{equation}
    \label{g evol}
    \dt g_{ij}=-2Fh_{ij}.
  \end{equation}
\end{lemma}
\begin{proof}
  We refer to \cite{os:alpbach}.
\end{proof}

\begin{corollary}
  The inverse metric $g^{ij}$ evolves according to 
  \begin{equation}
    \label{ig evol}
    \dt g^{ij}=2Fh^{ij}.
  \end{equation}
\end{corollary}
\begin{proof}
  Direct calculations yield
  \begin{align*}
  \dt g^{ij} =&\, -g^{ik} g^{ls} \underbrace{\dt g_{kl}}_{\text{see}\,\eqref{g evol}} \umbruch\\
                 =&\, 2 F g^{ik} h_{kl} g^{lj} \umbruch\\
                 =&\, 2 F h^{ij}.\qedhere
  \end{align*}
\end{proof}

\begin{lemma}
  The second fundamental form $h_{ij}$ evolves according to
  \begin{equation}
    \label{2ff evol}
    \begin{split}
      L\left(h_{ij}\right)=&\,F^{kl}h^a_kh_{al}\cdot h_{ij}-F^{kl}h_{kl}\cdot h^a_ih_{aj}\\
      &\,-Fh^k_ih_{kj}+F^{kl,\,rs}h_{kl;\,i}h_{rs;\,j}.
    \end{split}
  \end{equation}
\end{lemma}
\begin{proof}
We refer to \cite{os:alpbach}.
\end{proof}

\begin{lemma}
  The mean curvature $H$ evolves according to
  \begin{equation}
    \label{H evol}
    \begin{split}
      L\left(H\right)=&\,F^{kl} h^a_k h_{al} \cdot H + \left(F - F^{kl}h_{kl}\right) |A|^2+g^{ij}F^{kl,rs}h_{kl;i}h_{rs;j}.
    \end{split}
  \end{equation}
\end{lemma}
\begin{proof}
  This is a straightforward calculation.
  \begin{align*}
      &\, L\left(H\right) \umbruch\\
    =&\, \dt H - F^{kl}H_{;kl} \qquad \textbf{(\text{see}\;\eqref{def linear operator L})} \umbruch\\
    =&\,\dt\left(g^{ij} h_{ij}\right) - F^{kl} \big(\underbrace{H_{;k}}_{\text{see}\,\eqref{H diff}}\big)_{;l} \umbruch\\
    =&\,\underbrace{\left(\dt g^{ij}\right)}_{\text{see}\,\eqref{ig evol}} h_{ij} + g^{ij} \left( \dt h_{ij} \right)
      -F^{kl}\left(g^{ij} h_{ij;k}\right)_{;l} \umbruch\\
    =&\,2 F h^{ij} h_{ij} + g^{ij} \left(\dt h_{ij} \right) 
      - F^{kl} \bigg(\underbrace{\left(g^{ij}\right)_{;l}}_{=\,0} h_{ij;k} + g^{ij} h_{ij;kl} \bigg) \umbruch\\
    =&\,2 F |A|^2 + g^{ij} \underbrace{L\left(h_{ij}\right)}_{\text{see}\,\eqref{2ff evol}} \umbruch\\
    =&\,2 F |A|^2 \\
      &+g^{ij} \left( F^{kl}h^a_kh_{al}\cdot h_{ij}-F^{kl}h_{kl}\cdot h^a_ih_{aj}-Fh^k_ih_{kj}+F^{kl,\,rs}h_{kl;\,i}h_{rs;\,j}\right) \umbruch\\
    =&\,2 F |A|^2 \\
      &\,+ F^{kl} h^a_k h_{al} \cdot g^{ij} h_{ij} - F^{kl} h_{kl} \cdot h^a_i g^{ij} h_{aj} - F h^k_i g^{ij} h_{kj} + g^{ij} F^{kl,rs} h_{kl;i} h_{rs;j} \umbruch\\
    =&\,2 F |A|^2 \\
      &\,+F^{kl} h^a_k h_{al} \cdot H - F^{kl} h_{kl} \cdot |A|^2 - F |A|^2 + g^{ij} F^{kl,rs} h_{kl;i} h_{rs;j} \umbruch\\
    =&\,F^{kl} h^a_k h_{al} \cdot H + \left(F - F^{kl}h_{kl}\right) |A|^2+g^{ij}F^{kl,rs}h_{kl;i}h_{rs;j}.\qedhere
  \end{align*}
\end{proof}

\begin{lemma}
  The Gauss curvature $K$ evolves according to
  \begin{equation}
    \label{K evol}
    \begin{split}
      &\,L\left(K\right) \\
   =&\,K\left( F^{kl} h^a_k h_{al} \cdot \tilde{h}^{ij} h_{ij} - \left(F^{kl} h_{kl} + F\right) \tilde{h}^{ij} h^a_i h_{aj} \right. \\
     &\qquad\left. +2FH + \left( \tilde{h}^{ij} F^{kl,rs} + F^{ij} \left( \tilde{h}^{kr} \tilde{h}^{ls} - \tilde{h}^{kl} \tilde{h}^{rs} \right) \right) h_{kl;i} h_{rs;j} \right).
    \end{split}
  \end{equation}
\end{lemma}

\begin{proof}
  This is a straightforward calculation.
  \begin{align*}
    &\, L\left(K\right) \umbruch\\
  =&\, \dt K - F^{kl}K_{;kl} \qquad \textbf{(\text{see}\;\eqref{def linear operator L})} \umbruch\\
  =&\, \dt \left(\frac{\det h_{ij}}{\det g_{ij}} \right) - F^{kl}\big( \underbrace{K_{;k}}_{\text{see}\,\eqref{K diff}} \big)_{;l} \umbruch\\
  =&\, \frac{1}{\left( \det g_{ij} \right)^2} \left( \left( \dt \det h_{ij} \right) \left( \det g_{ij} \right) - \left( \det h_{ij} \right) \left( \dt \det g_{ij} \right) \right) \\
     &\, -F^{kl} \left( K \tilde{h}^{ij} h_{ij;k} \right)_{;l} \umbruch\\
  =&\, \frac{1}{\left( \det g_{ij} \right)^2 } \Bigg( \left( \frac{\partial}{\partial h_{ij}} \det h_{ij} \right) \left(\dt h_{ij} \right) \left(\det g_{ij} \right) \Bigg. \\
     &\qquad \qquad \qquad 
     \Bigg. -\left( \det h_{ij} \right) \left( \frac{\partial}{\partial g_{ij}} \det g_{ij} \right) \bigg( \underbrace{\dt g_{ij} }_{\text{see}\,\eqref{g evol}} \bigg) \Bigg) \\
     &\, -F^{kl} \left( K_{;l} \tilde{h}^{ij} h_{ij;k} + K \left(\tilde{h}^{ij} \right)_{;l} h_{ij;k} + K \tilde{h}^{ij} h_{ij;kl} \right) \umbruch\\
  =&\, K \left( \left(\tilde{h}^{ij} \right) \left(\dt h_{ij} \right) - \left( g^{ij} \right) \left( -2Fh_{ij} \right) \right) \\
     &\, -F^{kl} \left( K \tilde{h}^{rs} h_{rs;l} \tilde{h}^{ij} h_{ij;k} + K \left(-\tilde{h}^{ir} \tilde{h}^{sj} h_{rs;l} \right) h_{ij;k} + K \tilde{h}^{ij} h_{ij;kl} \right) \umbruch\\
  =&\, K \left( \tilde{h}^{ij} L\left( h_{ij} \right) + 2 F g^{ij} h_{ij} - F^{kl} \tilde{h}^{ij} \tilde{h}^{rs} h_{ij;k} h_{rs;l} + F^{kl} \tilde{h}^{ir} \tilde{h}^{sj} h_{ij;k} h_{rs;l} \right) \umbruch\\
  =&\, K \bigg( \tilde{h}^{ij} \underbrace{L\left( h_{ij} \right)}_{\text{see}\,\eqref{2ff evol}} + 2 F H + F^{ij} \left( \tilde{h}^{kr} \tilde{h}^{ls} - \tilde{h}^{kl} \tilde{h}^{rs} \right) h_{kl;i} h_{rs;j} \bigg) \umbruch\\
  =&\, K \bigg( \tilde{h}^{ij} \left(F^{kl} h^a_k h_{al} \cdot h_{ij} - F^{kl} h_{kl} \cdot h^a_i h_{aj} - F h^k_i h_{kj} + F^{kl,rs} h_{kl;i} h_{rs;j} \right) \bigg. \\
    &\qquad \bigg. +2 F H + F^{ij} \left( \tilde{h}^{kr} \tilde{h}^{ls} - \tilde{h}^{kl} \tilde{h}^{rs} \right) h_{kl;i} h_{rs;j} \bigg) \umbruch\\
  =&\, K \bigg( F^{kl} h^a_k h_{al} \cdot \tilde{h}^{ij} h_{ij} - F^{kl} h_{kl} \cdot \tilde{h}^{ij} h^a_i h_{aj} - F \tilde{h}^{ij} h^k_i h_{kj} + \tilde{h}^{ij} F^{kl,rs} h_{kl;i} h_{rs;j} \bigg. \\
    &\qquad \bigg. +2 F H + F^{ij} \left( \tilde{h}^{kr} \tilde{h}^{ls} - \tilde{h}^{kl} \tilde{h}^{rs} \right) h_{kl;i} h_{rs;j} \bigg) \umbruch\\
  =&\, K \bigg( F^{kl} h^a_k h_{al} \cdot \tilde{h}^{ij} h_{ij} - \left(F^{kl} h_{kl} + F\right) \tilde{h}^{ij} h^a_i h_{aj} \bigg. \\
    &\qquad \bigg. +2 F H + \left(\tilde{h}^{ij} F^{kl,rs} + F^{ij} \left( \tilde{h}^{rs} \tilde{h}^{ls} - \tilde{h}^{kl} \tilde{h}^{rs} \right) \right) h_{kl;i} h_{rs;j} \bigg).\qedhere
  \end{align*}
\end{proof}

\begin{lemma}
  The function $w\left(H,K\right)$ evolves according to
  \begin{equation}
    \label{w evol} 
    \begin{split}
        &\,L\big(w\left(H,K\right)\big) \\
      =&\,L\left(H\right)\,w_H + L\left(K\right)\,w_K \\
        &\,- F^{ij} \Big(H_{;i} H_{;j}\,w_{HH} + \left(H_{;i} K_{;j} + K_{;i} H_{;j}\right)\,w_{HK} + K_{;i} K_{;j}\,w_{KK} \Big),
    \end{split}    
  \end{equation}
where $H$ is the mean curvature and $K$ is the Gauss curvature. $H$ and $K$ form an algebraic basis of the symmetric homogeneous polynomials in two variables. \\
  
Here, the $w$-terms are defined as
\begin{align*}
  &\, w_H := \frac{\partial w}{\partial H},\;w_K := \frac{\partial w}{\partial K}, \umbruch\\
  &\, w_{HH} := \frac{\partial^2w}{\partial H^2},\; w_{HK} := \frac{\partial^2w}{\partial H \partial K},\;w_{KK} := \frac{\partial^2w}{\partial K^2}.
\end{align*} 
  
\end{lemma}

\begin{proof}
  We use the chain rule.
  \begin{align*}
      &\, L\big(w\left(H,K\right)\big) \umbruch\\
    =&\, \frac{d}{dt} w\left(H,K\right) - F^{ij} w\left(H,K\right)_{;ij} \umbruch\\
    =&\, \frac{d}{dt} H\,w_H + \frac{d}{dt} K\,w_K - F^{ij} \left( H_{;i}\,w_H + K_{;i}\,w_K\right)_{;j} \umbruch\\
    =&\, \frac{d}{dt} H\,w_H + \frac{d}{dt} K\,w_K - F^{ij} \left( H_{;ij}\,w_H + K_{;ij}\,w_K \right)\\
      &\, - F^{ij} \Big(H_{;i} H_{;j}\,w_{HH} + H_{;i} K_{;j}\,w_{HK} + H_{;j} K_{;i}\,w_{HK} + K_{;i} K_{;j}\,w_{KK}\Big) \umbruch\\
    =&\, L\left(H\right)\,w_H + L\left(K\right)\,w_K \\
      &\, - F^{ij} \Big(H_{;i} H_{;j}\,w_{HH} + \left(H_{;i} K_{;j} + H_{;j} K_{;i}\right)\,w_{HK} + K_{;i} K_{;j}\,w_{KK}\Big). \qedhere
  \end{align*} 
\end{proof}

\subsection{Evolution equations at a critical point, where we also choose normal coordinates.}

In the second part of Section \ref{sec evolution equations}, we calculate the evolution equations of the following geometric quantities at a critical point of the general function $w\left(H,K\right)$, 
  where we also choose normal coordinates,
  \begin{itemize}
    \item mean curvature $H$,
    \item Gauss curvature $K$,
    \item and general function $w\left(H,K\right)$ depending on $H$ and $K$.
  \end{itemize}

\begin{lemma}
  The covariant derivatives of the second fundamental form $h_{ij}$ fulfill these identitities at a critical point of $w\left(H,K\right)$ \text{\textbf{(CP)}}, 
  i.e. $w\left(H,K\right)_{;i} = 0$ for $i=1,2$, where we also choose normal coordinates \text{\textbf{(NC)}}, i.e. the metric tensor equals the Kronecker delta, $g_{ij} = \delta_{ij}$, and $\left(h_{ij}\right)$ is diagonal, $\left(h_{ij}\right) = \diag\left(\lx,\ly\right)$,
  \begin{equation}
    \label{h111}
      h_{22;1} = -\frac{w_H + \ly w_K}{w_H + \lx w_K} \cdot h_{11;1} \equiv a_1 \cdot h_{11;1},
  \end{equation}
    \begin{equation}
    \label{h222}
      h_{11;2} = -\frac{w_H + \lx w_K}{w_H + \ly w_K} \cdot h_{22;2} \equiv a_2 \cdot h_{22;2}.
  \end{equation}
\end{lemma}

\begin{proof}
  Let $i=1$.
  \begin{align*}
      &\, w\left(H,K\right)_{;1} \umbruch\\
    =&\, \underbrace{H_{;1}}_{\text{see}\,\eqref{H diff}} w_H + \underbrace{K_{;1}}_{\text{see}\,\eqref{K diff}} w_K \umbruch\\
    =&\, \left(g^{ij} h_{ij;1}\right) w_H + \left(K \tilde{h}^{ij} h_{ij;1} \right) w_K \umbruch\\
    =&\, \left(h_{11;1} + h_{22;1} \right) w_H + K \left( \frac{1}{\lx} h_{11;1} + \frac{1}{\ly} h_{22;1} \right) w_K \qquad \text{\textbf{(NC)}} \umbruch\\
    =&\, \left(w_H + \ly w_K \right) h_{11;1} + \left(w_H + \lx w_K\right) h_{22;1} \umbruch\\
    =&\, 0 \qquad \text{\textbf{(CP)}}.
  \end{align*}
  This implies
  \begin{align*}
    h_{22;1} = -\frac{w_H + \ly w_K}{w_H + \lx w_K} \cdot h_{11;1} \equiv a_1 \cdot h_{11;1}.
  \end{align*}
  Let $i=2$.
  \begin{align*}
      &\, w\left(H,K\right)_{;2} \\
    =&\, \underbrace{H_{;2}}_{\text{see}\,\eqref{H diff}} w_H + \underbrace{K_{;2}}_{\text{see}\,\eqref{K diff}} w_K \umbruch\\
    =&\, \left(g^{ij} h_{ij;2}\right) w_H + \left(K \tilde{h}^{ij} h_{ij;2} \right) w_K \umbruch\\
    =&\, \left(h_{11;2} + h_{22;2} \right) w_H + K \left( \frac{1}{\lx} h_{11;2} + \frac{1}{\ly} h_{22;2} \right) w_K \qquad \text{\textbf{(NC)}} \umbruch\\
    =&\, \left(w_H + \ly w_K \right) h_{11;2} + \left(w_H + \lx w_K\right) h_{22;2} \umbruch\\
    =&\, 0 \qquad \text{\textbf{(CP)}}.
  \end{align*}
  This implies
  \begin{align*}
    h_{11;2} = -\frac{w_H + \lx w_K}{w_H + \ly w_K} \cdot h_{22;2} \equiv a_2 \cdot h_{22;2}. 
  \end{align*}
\end{proof}

\begin{lemma}
  The covariant derivatives of the mean curvature $H$ \eqref{H diff} fulfill these identities at a critical point of $w\left(H,K\right)$, where we also choose normal coordinates,
    \begin{equation}
     \label{H diff ident 1}
      H_{;1} = \left(1 + a_1\right) h_{11;1},
    \end{equation}
    \begin{equation}
    \label{H diff ident 2}
      H_{;2} = \left(1 + a_2\right) h_{22;2}.
    \end{equation}
\end{lemma}
\begin{proof}
  \begin{align*}
    H_{;1} =&\, g^{ij}h_{ij;1} \qquad \textbf{(\text{see}\;\eqref{H diff})} \umbruch\\
              =&\, h_{11;1} + h_{22;1} \qquad \text{\textbf{(NC)}} \umbruch\\
              =&\, \left(1 + a_1\right) h_{11;1} \qquad \text{\textbf{(CP)}}, 
  \end{align*}
  \begin{align*}
    H_{;2} =&\, g^{ij}h_{ij;2} \qquad \textbf{(\text{see}\,\eqref{H diff})} \umbruch\\
              =&\, h_{11;2} + h_{22;2} \qquad \text{\textbf{(NC)}} \umbruch\\
              =&\, \left(1 + a_2\right) h_{22;2} \qquad \text{\textbf{(CP)}}.                                                                       
  \end{align*}
\end{proof}

\begin{lemma}
  The covariant derivatives of the Gauss curvature $K$ \eqref{K diff} fulfill these identities at a critical point of $w\left(H,K\right)$, where we also choose normal coordinates,
    \begin{equation}
      \label{K diff ident 1}
      K_{;1} = \left(\ly + \lx a_1\right) h_{11;1},
    \end{equation}
    \begin{equation}
      \label{K diff ident 2}
      K_{;2} = \left(\lx + \ly a_2\right) h_{22;2}.
    \end{equation}
\end{lemma}
\begin{proof}
  \begin{align*}
    K_{;1} =&\, K\tilde{h}^{ij}h_{ij;1} \qquad \textbf{(\text{see}\,\eqref{K diff})} \umbruch\\
              =&\, K\left(\frac{1}{\lx} h_{11;1} + \frac{1}{\ly} h_{22;1}\right) \qquad \text{\textbf{(NC)}} \umbruch\\
              =&\, \ly h_{11;1} + \lx h_{22;1} \umbruch\\
              =&\, \left(\ly + \lx a_1\right) h_{11;1} \qquad \text{\textbf{(CP)}}, 
  \end{align*}
  \begin{align*}                                                                
    K_{;2} =&\, K \tilde{h}^{ij}h_{ij;2} \qquad \textbf{(\text{see}\,\eqref{K diff})} \umbruch\\
              =&\, K\left(\frac{1}{\lx} h_{11;2} + \frac{1}{\ly} h_{22;2}\right) \qquad \text{\textbf{(NC)}} \umbruch\\
              =&\, \ly h_{11;2} + \lx h_{22;2} \umbruch\\
              =&\, \left(\lx + \ly a_2\right) h_{22;2} \qquad \text{\textbf{(CP)}}.                                                                       
  \end{align*}
\end{proof}

\begin{lemma}
  The evolution equation of the mean curvature $H$ \eqref{H evol} fulfills this identity at a critical point of $w\left(H,K\right)$, where we also choose normal coordinates,
  \begin{equation}
    \label{H evol ident}
    L\left(H\right) = C_H\left(\lx,\ly\right) + G_H\left(\lx,\ly\right) h_{11;1}^2 + G_H\left(\ly,\lx\right) h_{22;2}^2,
  \end{equation}
  \qquad where
  \begin{equation}
    \label{H constant ident}
    \begin{split}
    C_H\left(\lx,\ly\right) =&\, F \left(\lx^2 + \ly^2\right) + \left( \frac{\partial F}{\partial \lx} - \frac{\partial F}{\partial \ly} \right) \left(\lx - \ly\right) \lx \ly,
                                        \text{ and}
    \end{split}
  \end{equation}
  \begin{equation}
    \label{H gradient ident}
    \begin{split}
    G_H\left(\lx,\ly\right) =\, \frac{\partial^2F}{\partial \lx^2} + 2 \frac{\partial^2F}{\partial \lx \partial \ly} a_1 + \frac{\partial^2F}{\partial \ly^2} a_1^2
                                         + 2 \frac{\frac{\partial F}{\partial \lx} - \frac{\partial F}{\partial \ly}}{\lx-\ly} a_1^2.
    \end{split}
  \end{equation}
\end{lemma}

\begin{proof}
  \begin{align*}
        &\, L\left(H\right) \umbruch\\
     = &\, F^{kl} h^a_k h_{al} \cdot H + \left(F - F^{kl} h_{kl} \right) |A|^2 + g^{ij} F^{kl,rs} h_{kl;i} h_{rs;j} \qquad \textbf{(\text{see}\,\eqref{H evol})} \umbruch\\
     = &\, \left( \frac{\partial F}{\partial \lx} \lx^2 + \frac{\partial F}{\partial \ly^2} \ly^2 \right) \left(\lx + \ly\right) \umbruch\\
        &\, + \left( F - \frac{\partial F}{\partial \lx} \lx - \frac{\partial F}{\partial \ly} \ly \right) \left(\lx^2 + \ly^2\right) \\
        &\, + \left( \sum^2_{i,j=1} \frac{\partial^2F}{\partial \lambda_i \partial \lambda_j} h_{ii;1} h_{jj;1} 
             + \sum^2_{\substack{i,j=1 \\ i\neq j}} \frac{\frac{\partial F}{\partial \lambda_i} - \frac{\partial F}{\partial \lambda_j}}{ \lambda_i - \lambda_j} h_{ij;1}^2 \right) \\
        &\, + \left( \sum^2_{i,j=1} \frac{\partial^2F}{\partial \lambda_i \partial \lambda_j} h_{ii;2} h_{jj;2} 
             + \sum^2_{\substack{i,j=1 \\ i\neq j}} \frac{\frac{\partial F}{\partial \lambda_i} - \frac{\partial F}{\partial \lambda_j}}{ \lambda_i - \lambda_j} h_{ij;2}^2 \right) 
             \qquad \text{\textbf{(NC)}} \umbruch\\
     = &\, \left( \frac{\partial F}{\partial \lx} \lx^2 + \frac{\partial F}{\partial \ly^2} \ly^2 \right) \left(\lx + \ly\right) \\
        &\, + \left( F - \frac{\partial F}{\partial \lx} \lx - \frac{\partial F}{\partial \ly} \ly \right) \left(\lx^2 + \ly^2\right) \\
        &\, + \left( \frac{\partial^2F}{\partial \lx^2} h_{11;1}^2 + 2 \frac{\partial^2F}{\partial \lx \partial \ly} h_{11;1} h_{22;1} + \frac{\partial^2F}{\partial \ly^2} h_{22;1}^2
              + 2 \frac{\frac{\partial F}{\partial \lx} - \frac{\partial F}{\partial \ly}}{\lx - \ly} h_{12;1}^2 \right) \\
        &\, + \left( \frac{\partial^2F}{\partial \lx^2} h_{11;2}^2 + 2 \frac{\partial^2F}{\partial \lx \partial \ly} h_{11;2} h_{22;2} + \frac{\partial^2F}{\partial \ly^2} h_{22;2}^2
              + 2 \frac{\frac{\partial F}{\partial \lx} - \frac{\partial F}{\partial \ly}}{\lx - \ly} h_{12;2}^2 \right).
  \end{align*}
  
According to \cite{cg:curvature}, the terms 
\begin{align*}
  F^{ij,\,kl}\eta_{ij}\eta_{kl}=\sum^2_{i,j=1}\fracp{^2F}{\li\partial\lj}\eta_{ii}\eta_{jj}+\sum^2_{\substack{i,j=1 \\ i\neq j}}\frac{\fracp F\li-\fracp F\lj}{\li-\lj}(\eta_{ij})^2
\end{align*}
are well-defined for symmetric matrices $(\eta_{ij})$ and $\lx\neq\ly$ or $\lx=\ly$, when we interpret the last term as a limit. \\
  
We get the constant terms
  \begin{align*}
    C_H\left(\lx,\ly\right) =&\, F \left(\lx^2 + \ly^2\right) + \left( \frac{\partial F}{\partial \lx} - \frac{\partial F}{\partial \ly} \right) \left(\lx - \ly\right) \lx \ly,
  \end{align*}
  and at a critical point of $w\left(H,K\right)$ \textbf{(CP)}, using identities \eqref{h111} and \eqref{h222}, we get the gradient terms
  \begin{align*}
    \begin{split}
    G_H\left(\lx,\ly\right) =\, \frac{\partial^2F}{\partial \lx^2} + 2 \frac{\partial^2F}{\partial \lx \partial \ly} a_1 + \frac{\partial^2F}{\partial \ly^2} a_1^2
                                         + 2 \frac{\frac{\partial F}{\partial \lx} - \frac{\partial F}{\partial \ly}}{\lx-\ly} a_1^2.
    \end{split}
  \end{align*}
\end{proof}

\begin{lemma}
  The evolution equation of the Gauss curvature $K$ \eqref{K evol} fulfills this identity at a critical point of $w\left(H,K\right)$, 
  where we also choose normal coordinates,
  \begin{equation}
    \label{K evol ident}
    L\left( K \right) = C_K\left(\lx,\ly\right) + G_K\left(\lx,\ly\right) h_{11;1}^2 + G_K\left(\ly,\lx\right) h_{22;2}^2,
  \end{equation}
  \qquad where
  \begin{equation}\label{K constant ident}
    C_K\left(\lx,\ly\right) =\, \left( F\left(\lx+\ly\right) + \left( \frac{\partial F}{\partial \lx}  \lx - \frac{\partial F}{\partial \ly} \ly \right) \left( \lx - \ly \right) \right) \lx \ly,
  \end{equation}
  \begin{equation}\label{K gradient ident}
    \begin{split}
      G_K\left(\lx,\ly\right) =&\, 2\left(-\frac{\partial F}{\partial \lx} a_1 + \frac{\partial F}{\partial \ly} a_1^2 \right) \\
                                          & + \left( \frac{\partial^2F}{\partial \lx^2} + 2 \frac{\partial^2F}{\partial \lx \partial \ly} a_1 + \frac{\partial^2F}{\partial \ly^2} a_1^2 \right) \ly
                                          + 2 \frac{\frac{\partial F}{\partial \lx} - \frac{\partial F}{\partial \ly}}{\lx - \ly} a_1^2 \lx.
    \end{split}
  \end{equation}
\end{lemma}

\begin{proof}
  \begin{align*}
      &\, L\left( K \right) \umbruch\\
    =&\, K \bigg( F^{kl} h^a_k h_{al} \cdot \tilde{h}^{ij} h_{ij} - \left(F^{kl} h_{kl} + F\right) \tilde{h}^{ij} h^a_i h_{aj} + 2FH \bigg. \\
      &\, + \bigg. \left( \tilde{h}^{ij} F^{kl,rs} + F^{ij} \left(\tilde{h}^{kr} \tilde{h}^{ls} - \tilde{h}^{kl} \tilde{h}^{rs} \right) \right) h_{kl;i} h_{rs;j} \bigg) 
            \qquad \text{see}\,\eqref{K evol}\umbruch\\
    =&\, \lx\ly \Bigg( 2\left( \frac{\partial F}{\partial \lx} \lx^2 + \frac{\partial F}{\partial \ly} \ly^2 \right) + 
            \left(F - \frac{\partial F}{\partial \lx} \lx - \frac{\partial F}{\partial \ly} \ly\right) \left(\lx+\ly\right)   \Bigg. \\
      &\, \qquad + \frac{1}{\lx} \left(\sum^2_{i,j=1} \frac{\partial^2F}{\partial \lambda_i \partial \lambda_j} h_{ii;1} h_{jj;1} 
            + \sum^2_{\substack{i,j=1 \\ i\neq j}} \frac{ \frac{\partial F}{\partial \lambda_i} 
            - \frac{\partial F}{\partial \lambda_j} }{\lambda_i - \lambda_j} h_{ij;1}^2 \right) \\
      &\, \qquad + \frac{1}{\ly} \left(\sum^2_{i,j=1} \frac{\partial^2F}{\partial \lambda_i \partial \lambda_j} h_{ii;2} h_{jj;2} 
            + \sum^2_{\substack{i,j=1 \\ i\neq j}} \frac{ \frac{\partial F}{\partial \lambda_i} 
            - \frac{\partial F}{\partial \lambda_j} }{\lambda_i - \lambda_j} h_{ij;2}^2 \right) \\
      &\, \qquad + \frac{\partial F}{\partial \lx} \left( 2 \frac{1}{\lx \ly} h_{11;2}^2 \right)
                         - \frac{\partial F}{\partial \lx} \left( 2 \frac{1}{\lx \ly} h_{11;1} h_{22;1} \right) \\
      &\, \qquad + \Bigg. \frac{\partial F}{\partial \ly} \left( 2 \frac{1}{\lx \ly} h_{22;1}^2 \right)
                         - \frac{\partial F}{\partial \ly} \left( 2 \frac{1}{\lx \ly} h_{11;2} h_{22;2} \right) \Bigg) \qquad \textbf{(\text{NC})} \umbruch\\
    =&\, \lx\ly \Bigg( F \left(\lx+\ly\right) + \left( \frac{\partial F}{\partial \lx} \lx + \frac{\partial F}{\partial \ly} \ly \right) \left(\lx - \ly\right) \Bigg. \\ 
      &\, \qquad + \frac{1}{\lx} \left( \frac{\partial^2F}{\partial \lx^2} h_{11;1}^2 + 2 \frac{\partial^2F}{\partial \lx \partial \ly} h_{11;1} h_{22;1} 
                        + \frac{\partial^2F}{\partial \ly^2} h_{22;1}^2 + \frac{\frac{\partial F}{\partial \lx} - \frac{\partial F}{\partial \ly}}{\lx - \ly} h_{12;1}^2 \right) \\
      &\, \qquad + \frac{1}{\ly} \left( \frac{\partial^2F}{\partial \lx^2} h_{11;2}^2 + 2 \frac{\partial^2F}{\partial \lx \partial \ly} h_{11;2} h_{22;2} 
                        + \frac{\partial^2F}{\partial \ly^2} h_{22;2}^2 + \frac{\frac{\partial F}{\partial \lx} - \frac{\partial F}{\partial \ly}}{\lx - \ly} h_{12;2}^2 \right) \\
      &\, \qquad + \Bigg. 2 \frac{1}{\lx \ly} \left( \frac{\partial F}{\partial \lx} \left( h_{11;2}^2 - h_{11;1} h_{22;1} \right) 
                                                                      + \frac{\partial F}{\partial \ly} \left( h_{22;1}^2 - h_{11;2} h_{22;2} \right) \right) \Bigg).
  \end{align*}
  
  According to \cite{cg:curvature}, the terms 
  \begin{align*}
    F^{ij,\,kl}\eta_{ij}\eta_{kl}=\sum^2_{i,j=1}\fracp{^2F}{\li\partial\lj}\eta_{ii}\eta_{jj}+\sum^2_{\substack{i,j=1 \\ i\neq j}}\frac{\fracp F\li-\fracp F\lj}{\li-\lj}(\eta_{ij})^2
  \end{align*}
  are well-defined for symmetric matrices $(\eta_{ij})$ and $\lx\neq\ly$ or $\lx=\ly$, when we interpret the last term as a limit. \\
  
  We get the constant terms
  \begin{align*}
    C_K\left(\lx,\ly\right) =\, \left( F\left(\lx+\ly\right) + \left( \frac{\partial F}{\partial \lx}  \lx - \frac{\partial F}{\partial \ly} \ly \right) \left( \lx - \ly \right) \right) \lx \ly,
  \end{align*}
  and at a critical point of $w\left(H,K\right)$ \textbf{(CP)}, using identities \eqref{h111} and \eqref{h222}, we get the gradient terms
  \begin{align*}
    G_K\left(\lx,\ly\right) =&\, 2\left(-\frac{\partial F}{\partial \lx} a_1 + \frac{\partial F}{\partial \ly} a_1^2 \right) \\
                                          & + \left( \frac{\partial^2F}{\partial \lx^2} + 2 \frac{\partial^2F}{\partial \lx \partial \ly} a_1 + \frac{\partial^2F}{\partial \ly^2} a_1^2 \right) \ly
                                          + 2 \frac{\frac{\partial F}{\partial \lx} - \frac{\partial F}{\partial \ly}}{\lx - \ly} a_1^2 \lx.
  \end{align*}
\end{proof}

\begin{lemma}\label{lem evolution equation w}
  The evolution equation of the function $w\left(H,K\right)$ \eqref{w evol} fulfills this identity at a critical point of $w\left(H,K\right)$, 
  where we also choose normal coordinates,
  \begin{equation}\label{w evol ident}
    L\big(w\left(H,K\right)\big) = C_w\left(\lx,\ly\right) + G_w\left(\lx,\ly\right) h_{11;1}^2 + G_w\left(\lx,\ly\right) h_{22;2}^2,
  \end{equation}
  \qquad where
  \begin{equation}\label{w constant ident}
    \begin{split}
      C_w\left(\lx,\ly\right) =&\, C_H\left(\lx,\ly\right) w_H + C_K\left(\lx,\ly\right) w_K \\
                                        =&\, \left(F \left(\lx^2 + \ly^2\right) + \left( \frac{\partial F}{\partial \lx} 
                                                - \frac{\partial F}{\partial \ly} \right) \left(\lx - \ly\right) \lx \ly \right) w_H \\
                                          &\, + \left( F\left(\lx+\ly\right) + \left( \frac{\partial F}{\partial \lx}  \lx - \frac{\partial F}{\partial \ly} \ly \right) 
                                                \left( \lx - \ly \right) \right) \lx \ly w_K,
    \end{split}
  \end{equation}
  \begin{equation}\label{w gradient ident}
    \begin{split}
        &\,G_w\left(\lx,\ly\right) \\
      =&\, G_H\left(\lx,\ly\right) w_H + G_K\left(\lx,\ly\right) w_K \\
        &\, -\frac{\partial F}{\partial \lx} \left(\left(1+a_1\right)^2 w_{HH} + 2\left(1+a_1\right)\left(\ly+\lx a_1\right) w_{HK} 
           + \left(\ly+\lx a_1\right)^2 w_{KK} \right) \\
      =&\, \Bigg( \left( \frac{\partial^2F}{\partial \lx^2} + 2 \frac{\partial^2F}{\partial \lx \partial \ly} a_1 + 
                           \frac{\partial^2F}{\partial \ly^2} a_1^2 \right)
             + 2 \frac{\frac{\partial F}{\partial \lx} - \frac{\partial F}{\partial \ly}}{\lx - \ly} a_1^2 \Bigg) w_H \\
        &\, + \Bigg( 2\left(-\frac{\partial F}{\partial \lx} a_1 + \frac{\partial F}{\partial \ly} a_1^2 \right) \Bigg. \\
        &\, \qquad + \Bigg. \left( \frac{\partial^2F}{\partial \lx^2} + 2 \frac{\partial^2F}{\partial \lx \partial \ly} a_1 
              + \frac{\partial^2F}{\partial \ly^2} a_1^2 \right) \ly
              + 2 \frac{\frac{\partial F}{\partial \lx} - \frac{\partial F}{\partial \ly}}{\lx - \ly} a_1^2 \lx \Bigg) w_K \\
        &\, -\frac{\partial F}{\partial \lx} \left(\left(1+a_1\right)^2 w_{HH} + 2\left(1+a_1\right)\left(\ly+\lx a_1\right) w_{HK} 
              + \left(\ly+\lx a_1\right)^2 w_{KK} \right). 
    \end{split}
  \end{equation}
  Here, the $w$-terms are defined as
  \begin{align*}
    &\, w_H := \frac{\partial w}{\partial H},\;w_K := \frac{\partial w}{\partial K}, \\
    &\, w_{HH} := \frac{\partial^2 w}{\partial H^2},\;w_{HK} := \frac{\partial^2 w}{\partial H \partial K},\;w_{KK} := \frac{\partial^2 w}{\partial K^2}.
  \end{align*}
\end{lemma}
\begin{proof}
  \begin{align*}
      &\, L\left(w\left(H,K\right)\right) \\
    =&\, L\left(H\right) w_H + L\left(K\right) w_K \\
      &\,- F^{ij} \left(H_{;i} H_{;j} w_{HH} + \left(H_{;i} K_{;j} + K_{;i} H_{;j}\right) w_{HK} + K_{;i} K_{;j} w_{KK} \right) \qquad \textbf{(\text{see}\,\eqref{w evol})} \umbruch\\
    =&\, L\left(H\right) w_H + L\left(K\right) w_K \\
      &\, -\frac{\partial F}{\partial \lx}\left(H_{;1}^2 w_{HH} + 2 H_{;1}K_{;1} w_{HK} + K_{;1}^2 w_{KK} \right) \\
      &\, -\frac{\partial F}{\partial \ly}\left(H_{;2}^2 w_{HH} + 2 H_{;2}K_{;2} w_{HK} + K_{;2}^2 w_{KK} \right) \qquad \text{\textbf{(NC)}} \umbruch\\
    =&\, \left(C_H\left(\lx,\ly\right) + G_H\left(\lx,\ly\right) h_{11;1}^2 + G_H\left(\lx,\ly\right) h_{22;2}^2 \right) w_H \\
      &\, + \left(C_K\left(\lx,\ly\right) + G_K\left(\lx,\ly\right) h_{11;1}^2 + G_K\left(\lx,\ly\right) h_{22;2}^2 \right) w_K \\
      &\, -\frac{\partial F}{\partial \lx} \left(\left(1+a_1\right)^2 w_{HH} + 2\left(1+a_1\right)\left(\ly+\lx a_1\right) w_{HK} 
           + \left(\ly+\lx a_1\right)^2 w_{KK} \right) h_{11;1}^2 \\
      &\, -\frac{\partial F}{\partial \ly} \left(\left(1+a_2\right)^2 w_{HH} + 2\left(1+a_2\right)\left(\lx+\ly a_2\right) w_{HK} 
           + \left(\lx+\ly a_2\right)^2 w_{KK} \right) h_{22;2}^2 \\
      &\, \qquad  \text{\textbf{(see \eqref{H diff ident 1}, \eqref{H diff ident 2}, \eqref{K diff ident 1}, \eqref{K diff ident 2}, \eqref{H evol ident}, \eqref{K evol ident})}}.
  \end{align*}
  We get the constant terms
  \begin{align*}
    C_w\left(\lx,\ly\right) =&\, C_H\left(\lx,\ly\right) w_H + C_K\left(\lx,\ly\right) w_K \umbruch\\
                                        =&\, \left(F \left(\lx^2 + \ly^2\right) + \left( \frac{\partial F}{\partial \lx} 
                                                - \frac{\partial F}{\partial \ly} \right) \left(\lx - \ly\right) \lx \ly \right) w_H \umbruch\\
                                          &\, + \left( F\left(\lx+\ly\right) + \left( \frac{\partial F}{\partial \lx}  \lx - \frac{\partial F}{\partial \ly} \ly \right) 
                                                \left( \lx - \ly \right) \right) \lx \ly w_K \\
                                          &\, \qquad \text{\textbf{(see \eqref{H constant ident}, \eqref{K constant ident})}},
  \end{align*}
  and at a critical point of $w\left(H,K\right)$ \textbf{(CP)}, we get the gradient terms
  \begin{align*}
    \begin{split}
        &\,G_w\left(\lx,\ly\right) \\
      =&\, G_H\left(\lx,\ly\right) w_H + G_K\left(\lx,\ly\right) w_K \\
        &\, -\frac{\partial F}{\partial \lx} \left(\left(1+a_1\right)^2 w_{HH} + 2\left(1+a_1\right)\left(\ly+\lx a_1\right) w_{HK} 
           + \left(\ly+\lx a_1\right)^2 w_{KK} \right) \\
      =&\, \Bigg( \left( \frac{\partial^2F}{\partial \lx^2} + 2 \frac{\partial^2F}{\partial \lx \partial \ly} a_1 + 
                           \frac{\partial^2F}{\partial \ly^2} a_1^2 \right)
             + 2 \frac{\frac{\partial F}{\partial \lx} - \frac{\partial F}{\partial \ly}}{\lx - \ly} a_1^2 \Bigg) w_H \\
        &\, + \Bigg( 2\left(-\frac{\partial F}{\partial \lx} a_1 + \frac{\partial F}{\partial \ly} a_1^2 \right) \Bigg. \\
        &\, \qquad + \Bigg. \left( \frac{\partial^2F}{\partial \lx^2} + 2 \frac{\partial^2F}{\partial \lx \partial \ly} a_1 
              + \frac{\partial^2F}{\partial \ly^2} a_1^2 \right) \ly
              + 2 \frac{\frac{\partial F}{\partial \lx} - \frac{\partial F}{\partial \ly}}{\lx - \ly} a_1^2 \lx \Bigg) w_K \\
        &\, -\frac{\partial F}{\partial \lx} \left(\left(1+a_1\right)^2 w_{HH} + 2\left(1+a_1\right)\left(\ly+\lx a_1\right) w_{HK} 
              + \left(\ly+\lx a_1\right)^2 w_{KK} \right) \\
        &\, \qquad \text{\textbf{(see \eqref{H gradient ident}, \eqref{K gradient ident})}}.
    \end{split}
  \end{align*}
\end{proof}

\subsection{Constant terms and gradient terms at a critical point, where we also choose normal coordinates. The normal velocity is equal to powers of the Gauss curvature.}

  In the third and last part of Section \ref{sec evolution equations}, we calcuate the constant terms $C_w\left(\lx,\ly\right)$ and the gradient terms 
   $G_w\left(\lx,\ly\right)$ from Definition \ref{def mpf} of the \textit{maximum-principle functions}.
  As before, we calculate these terms at a critical point of the general function $w\left(H,K\right)$, 
  where we also choose normal coordinates. Furthermore, we set the normal velocity to powers of the Gauss curvature, $F=K^{\sigma}$.
  
  So far, the constant terms and the gradient terms are rational functions. 
  Now we divide each of them by some nonnegative factor in order to turn them into polynomials in two variables.
  For the constant terms we get a symmetric polynomial and for the gradient terms an asymmetric polynomial.
  Finally, we dehomogenize both polynomials by setting $\lx=\rho$ and $\ly=1$ and vice versa.
  We obtain for the constant terms a polynomial in one variable $C\left(\rho\right)$ and 
  for the gradient terms two polynomials in one variable $G_1\left(\rho\right)$ and $G_2\left(\rho\right)$.

\begin{lemma} Calculating the evolution equation of the quotient of two functions, $w = \frac{p}{q}$, we obtain the following identity
  \begin{align}\label{pq evol}
    q^2\;L\left(\frac{p}{q}\right)=q\;L(p)-p\;L(q)
  \end{align}
  at a critical point of $w$, i.e. $w_{;i} = 0$ for $i = 1,\,2$.
\end{lemma}

\begin{proof}
  \begin{align*}
    L\left(\frac{p}{q}\right)=&\,\frac{d}{dt}\left(\frac{p}{q}\right)-F^{ij}\left(\frac{p}{q}\right)_{;ij}\\
    =&\,\frac{\dot{p}q-p\dot{q}}{q^2}-F^{ij}\frac{1}{q^4}\bigg(\left(p_i q-p q_i\right)_{;j} q^2 - \underbrace{\left(p_i q-p q_i\right)}_{=0\;\textbf{(CP)}}\left(q^2\right)_{;j}\bigg)\\
    =&\,
    \frac{1}{q^2}\Bigg(\dot{p} q -p \dot{q} - F^{ij}\bigg(p_{;ij}\,q + \underbrace{p_i q_j - p_j q_i}_{=0} - p\,q_{;ij}\bigg)\Bigg)\\
    =&\,\frac{1}{q^2}\bigg(q\,L(p) - p\,L(q)\bigg).
  \end{align*}
\end{proof}

\begin{lemma}\label{lem constant terms}
  We calculate the following constant terms at a critical point, where we also choose normal coordinates. The normal velocity is equal to powers of the Gauss curvature, 
  $F = K^\sigma$. \\
  
  First we calculate the constant terms for the mean curvature $H$
  \begin{align}
    C_H\left(\lx,\ly\right) =&\, K^\sigma\left(\lx^2 + \ly^2 - \sigma\left(\lx - \ly\right)^2\right),
  \end{align}
  
  and the constant terms for the Gauss curvature $K$
  \begin{align}
    C_K\left(\lx,\ly\right) =&\,  K^{\sigma+1}\left(\lx + \ly\right).
  \end{align}
  
  Then we calculate the constant terms for a rational function $w=\frac{p}{q}$
  \begin{align}
    \begin{split}
    q^2 C_{\frac{p}{q}}\left(\lx,\ly\right) =&\,C_H\left(\lx,\ly\right) r_H + C_K\left(\lx,\ly\right) r_K \\
                                                  =&\,K^\sigma\left(\left(\left(\lx^2+\ly^2\right)-\sigma\left(\lx-\ly\right)^2\right)r_H + K\left(\lx+\ly\right)r_K\right).
    \end{split}
  \end{align}
  
  Now we divide the previous constant terms by a nonnegative factor and get a polynomial in two variables
  \begin{align}
  C_r\left(\lx,\ly\right) :=&\, \frac{q^2}{K^\sigma} C_{\frac{p}{q}}\left(\lx,\ly\right).
  \end{align}
  
  We dehomogenize the previous polynomial setting $\lx=\rho,\,\ly=1$ and get a polynomial in one variable 
  \begin{align}
  C\left(\rho\right) :=&\,C_r\left(\rho,1\right) = \left(\left(1-\sigma\right)\rho^2+2\sigma\rho+\left(1-\sigma\right)\right)r_H + \rho\left(\rho+1\right)r_K.
  \end{align}
  
 Here, the $r$-terms are defined as
 \begin{align*}
    &\, r_H:=q\;\frac{\partial p}{\partial H}-p\;\frac{\partial q}{\partial H}, \umbruch\\
    &\, r_K:=q\;\frac{\partial p}{\partial K}-p\;\frac{\partial q}{\partial K}.
  \end{align*}
  
\end{lemma}
\begin{proof}
  We calculate the constant terms $C_H\left(\lx,\ly\right)$ for $F=K^\sigma$.
  \begin{align*}
      &\, C_H\left(\lx,\ly\right) \umbruch\\
    =&\, F\left(\lx^2+\ly^2\right) + \left(\frac{\partial F}{\partial \lx} - \frac{\partial F}{\partial \ly}\right) \left(\lx - \ly\right) \lx \ly 
            \qquad \textbf{(\text{see}\,\eqref{H constant ident})} \umbruch\\
    =&\, K^\sigma \left(\lx^2 + \ly^2\right) + \left(\sigma K^{\sigma-1}\ly-\sigma K^{\sigma-1}\lx\right)\left(\lx-\ly\right)\lx \ly \umbruch\\
    =&\, K^\sigma \left(\lx^2 + \ly^2 - \sigma \left( \lx - \ly \right)^2 \right).
  \end{align*}
  We calculate the constant terms $C_K\left(\lx,\ly\right)$ for $F=K^\sigma$.
  \begin{align*}
      &\, C_K\left(\lx,\ly\right) \umbruch\\
    =&\, \left( F\left(\lx+\ly\right) + \left( \frac{\partial F}{\partial \lx} \lx - \frac{\partial F}{\partial \ly} \ly \right) \left(\lx - \ly\right)\right)\lx \ly 
            \qquad \textbf{(\text{see}\,\eqref{K constant ident})} \umbruch\\
    =&\, \left( K^\sigma \left(\lx + \ly\right) + \left(\sigma K^{\sigma-1} \lx \ly - \sigma K^{\sigma-1} \lx \ly\right)\left(\lx-\ly\right)\right) \lx \ly \umbruch\\
    =&\, K^{\sigma+1} \left(\lx + \ly\right) \lx \ly.
  \end{align*}
  Using the identity $q^2 L\left(\frac{p}{q}\right) = q\,L\left(p\right) - p\,L \left(q\right)$ \eqref{pq evol} we calculate the constant terms $C_{\frac{p}{q}}\left(\lx,\ly\right)$.
  \begin{align*}
      &\, q^2 C_{ \frac{p}{q} } \left(\lx,\ly\right) \umbruch\\
    =&\, C_H\left(\lx,\ly\right) r_H + C_K\left(\lx,\ly\right) r_K 
            \qquad \textbf{(\text{see}\,\eqref{pq evol},\;\eqref{w constant ident})} \umbruch\\
    =&\, K^\sigma \left(\left(\lx^2 + \ly^2 - \sigma \left( \lx - \ly \right)^2 \right) r_H + K \left(\lx + \ly\right) r_K \right).
  \end{align*}
  Dividing by the nonnegative factor $\frac{K^\sigma}{q^2}$ we get the polynomial in two variables $C_r\left(\lx,\ly\right)$.
  \begin{align*}
    C_r\left(\lx,\ly\right)=&\, \frac{q^2}{K^{\sigma}} C_{\frac{p}{q}}\left(\lx,\ly\right) \umbruch\\
                                    =&\, \left(\lx^2 + \ly^2 - \sigma \left( \lx - \ly \right)^2 \right) r_H + K \left(\lx + \ly\right) r_K.
  \end{align*}
 Now we dehomogenize the previous polynomial setting $\lx=\rho,\,\ly=1$. We get the polynomial in one veriable $C\left(\rho\right)$.
  \begin{align*}
      C\left(\rho\right) :=&\, C_r\left(\rho,1\right) \umbruch\\
                                   =&\, \left(\rho^2 + 1 - \sigma\left(\rho-1\right)^2\right) r_H + \rho \left(\rho + 1\right) r_K \umbruch\\
                                   =&\, \left( \left( 1-\sigma \right) \rho^2 + 2\sigma \rho + \left(1-\sigma\right) \right) r_H + \rho \left(\rho + 1\right) r_K.
  \end{align*}
\end{proof}

\begin{lemma}\label{lem gradient terms}
  We calculate the following gradient terms at a critical point, where we also choose normal coordinates. The normal velocity is equal to powers of the Gauss curvature, $F=K^\sigma$. \\
  
  First we calculate the gradient terms for the mean curvature $H$
  \begin{align}\label{H gradient Ksigma}
    \begin{split}
    G_H\left(\lx,\ly\right) =&\, \frac{\sigma K^{\sigma-2}}{\left(w_H + \lx w_K\right)^2} \bigg(\left(-\left(\lx+\ly\right)^2 + \sigma \left(\lx - \ly\right)^2 \right) w_H^2 \bigg. \\
                                         &\, \qquad \bigg. -2\left(\lx + 3\ly\right) K w_H w_K - 2 \left(\lx + \ly\right) \ly K w_K^2 \bigg),
    \end{split}
  \end{align}
  
  the gradient terms for the Gauss curvature $K$
  \begin{align}\label{K gradient Ksigma}
    G_K\left(\lx,\ly\right) =&\, \frac{\sigma K^{\sigma-2}}{\left(w_H + \lx w_K\right)^2} \left(\sigma-1\right)\left(\lx-\ly\right)^2\ly w_H^2,
  \end{align}
  
  and the mixed terms (compare \eqref{w gradient ident})
  \begin{align}\label{mixed terms Ksigma}
    \begin{split}
        &\, -\frac{\partial F}{\partial \lx} \left( \left(1+a_1\right)^2 w_{HH} + 2\left(1+a_1\right)\left(\ly+\lx a_1\right) w_{HK} + \left(\ly + \lx a_1\right)^2 w_{KK} \right) \\
      =&\, \frac{\sigma K^{\sigma-2}}{\left(w_H + \lx w_K\right)^2} \lx \left(\lx-\ly\right)^2\ly^2 \left(-w_K^2 w_{HH} + 2 w_H w_K w_{HK} -w_H^2 w_{KK} \right).
    \end{split}
  \end{align}
  
  Then we calculate the gradient terms for a rational function $w=\frac{p}{q}$
  \begin{align}
    \begin{split}
        &\, q^2 G_{\frac{p}{q}}\left(\lx,\ly\right) \\
      =&\,G_H\left(\lx,\ly\right) r_H + G_K\left(\lx,\ly\right) r_K \\
         &\,-\frac{\partial F}{\partial \lx} \left( \left(1+a_1\right)^2 w_{HH} + 2\left(1+a_1\right)\left(\ly+\lx a_1\right) w_{HK} + \left(\ly + \lx a_1\right)^2 w_{KK} \right) \\
      =&\, \frac{\sigma K^{\sigma-2}}{\left(r_H + \lx r_K\right)^2} 
                                                                                              \bigg(\left( \left(\sigma-1\right)\lx^2-2\left(\sigma+1\right)\lx \ly + \left(\sigma-1\right)\ly^2 \right) r_H^3 \bigg. \\
                                         &\, \qquad \qquad \qquad \quad \bigg. \left( \left(\sigma-3\right) \lx^2 - 2\left(\sigma+2\right) \lx \ly + \left(\sigma-1\right) \ly^2 \right) \ly r_H r_K^2 \bigg. \\
                                         &\, \qquad \qquad \qquad \quad \bigg. -2 \lx \left(\lx+\ly\right) \ly^2 r_H r_K^2 \bigg. \\
                                         &\, \qquad \qquad \qquad \quad \bigg. -\lx \left(\lx-\ly\right)^2 \ly^2 r_K^2 r_{HH} \bigg. \\
                                         &\, \qquad \qquad \qquad \quad \bigg. +2\lx\left(\lx-\ly\right)^2\ly^2 r_H r_K r_{HK} \bigg. \\
                                         &\, \qquad \qquad \qquad \quad \bigg. -\lx\left(\lx-\ly\right)^2\ly^2r_H^2 r_{KK} \bigg).  
    \end{split}
  \end{align}
  
  Now we divide the previous gradient terms by a nonnegative factor and get a polynomial in two variables
  \begin{align}
  G_r\left(\lx,\ly\right) :=&\, \frac{q^2 \left(w_H + \lx w_K\right)^2}{\sigma K^{\sigma-2}} G_{\frac{p}{q}}\left(\lx,\ly\right).
  \end{align}
  
  We dehomogenize the previous polynomial setting $\lx=\rho,\,\ly=1$ and $\lx=1,\,\ly=\rho$, respectively. We get two polynomials in one variable 
  \begin{align}
    \begin{split}
      G_1\left(\rho\right) :=&\,G_r\left(\rho,1\right) \\
                                   =&\, \left( \left(\sigma-1\right)\rho^2-2\left(\sigma+1\right)\rho + \left(\sigma-1\right) \right) r_H^3  \\
                                          &\, + \left( \left(\sigma-3\right) \rho^2 - 2\left(\sigma+2\right) \rho + \left(\sigma-1\right) \right) r_H^2 r_K  \\
                                          &\, -2 \rho \left(\rho+1\right) r_H r_K^2  \\
                                          &\, -\rho \left(\rho-1\right)^2 r_K^2 r_{HH}  \\
                                          &\, +2\rho\left(\rho-1\right)^2 r_H r_K r_{HK}  \\
                                          &\, -\rho\left(\rho-1\right)^2 r_H^2 r_{KK},
    \end{split}
  \end{align}
  \begin{align}
    \begin{split}
      G_2\left(\rho\right) :=&\,G_r\left(1,\rho\right) \\
                                   =&\, \left( \left(\sigma-1\right)\rho^2-2\left(\sigma+1\right) \rho + \left(\sigma-1\right) \right) r_H^3  \\
                                          &\, + \rho \left( \left(\sigma-1\right)\rho^2 - 2\left(\sigma+2\right) \rho + \left(\sigma-3\right) \right) r_H^2 r_K  \\
                                          &\, -2 \rho^2 \left(\rho+1\right) r_H r_K^2  \\
                                          &\, -\rho^2 \left(\rho-1\right)^2 r_K^2 r_{HH}  \\
                                          &\, +2\rho^2\left(\rho-1\right)^2 r_H r_K r_{HK}  \\
                                          &\, -\rho^2\left(\rho-1\right)^2 r_H^2 r_{KK}.
    \end{split}
  \end{align}
  
 Here, the $r$-terms are defined as
 \begin{align*}
    &\, r_H:=q\;\frac{\partial p}{\partial H}-p\;\frac{\partial q}{\partial H},\;
    r_K:=q\;\frac{\partial p}{\partial K}-p\;\frac{\partial q}{\partial K}, \\
    &\, r_{HH}:=q\;\frac{\partial^2p}{\partial H^2}-p\;\frac{\partial^2q}{\partial H^2},\;
    r_{HK}:=q\;\frac{\partial^2p}{\partial H \partial K}-p\;\frac{\partial^2q}{\partial H \partial K},\;
    r_{KK}:=q\;\frac{\partial^2p}{\partial K^2}-p\;\frac{\partial^2q}{\partial K^2}.
  \end{align*}
  
\end{lemma}

\begin{proof}

We calculate the gradient terms $G_H\left(\lx,\ly\right)$ for $F=K^\sigma$.
  \begin{align*}
      & \, G_H\left(\lx,\ly\right) \umbruch\\
    =&\, \frac{\partial^2 F}{\partial \lx^2} + 2 \frac{\partial^2 F}{\partial \lx \partial \ly} a_1 
         + \frac{\partial^2 F}{\partial \ly^2} a_1^2 + \frac{\frac{\partial F}{\partial \lx} - \frac{\partial F}{\partial \ly}}{\lx - \ly} a_1^2 
         \qquad \textbf{(\text{see}\,\eqref{H gradient ident})} \umbruch\\
    =&\, \sigma\left(\sigma-1\right) K^{\sigma-2} \ly^2 + 2\sigma^2 K^{\sigma-1} a_1 + \sigma\left(\sigma-1\right) K^{\sigma-2} \lx^2 a_1^2 \\
      &\, + 2 \frac{\sigma K^{\sigma-1} \ly - \sigma K^{\sigma-1} \lx}{\lx - \ly} a_1^2 \umbruch\\
    =&\, \sigma K^{\sigma-2} \left( \left( \sigma-1\right) \ly^2 + 2\sigma K a_1 + \left(\sigma-1\right) \lx^2 a_1^2 - 2 K a_1^2 \right) \umbruch\\
    =&\, \sigma K^{\sigma-2} \Bigg( \Bigg. \left( \sigma-1 \right) \ly^2 + 2\sigma K \left(-\frac{w_H + \ly w_K}{w_H + \lx w_K} \right)  \\
      &\, \qquad \qquad + \left(\left(\sigma-1\right) \lx^2 - 2K\right)\left(-\frac{w_H + \ly w_K}{w_H + \lx w_K}\right)^2 \Bigg. \Bigg) 
      \qquad \textbf{(\text{see}\;\eqref{h111})} \umbruch\\
    =&\, \frac{\sigma K^{\sigma-2}}{\left( w_H + \lx w_K\right)^2} \left( \left(\sigma-1\right) \ly^2 \left(w_H + \lx w_K\right)^2 \right. \\
      &\,  \qquad \qquad \qquad \qquad - \left. 2\sigma K \left(w_H + \ly w_K \right) \left(w_H + \lx w_K\right) \right. \\
      &\,  \qquad \qquad \qquad \qquad + \left. \left( \left(\sigma - 1\right) \lx^2 - 2 K \right) \left(w_H + \ly w_K\right)^2 \right) \umbruch\\
    =&\, \frac{\sigma K^{\sigma-2}}{\left( w_H + \lx w_K\right)^2} \Big( \left(\sigma-1\right) \ly^2 \left(w_H^2 + 2\lx w_H w_K + \lx^2 w_K^2\right) \Big. \\
      &\,  \qquad \qquad \qquad \qquad - \Big. 2\sigma K \left(w_H^2 + \lx w_H w_K + \ly w_H w_K + K w_K^2\right) \Big. \\
      &\,  \qquad \qquad \qquad \qquad + \Big. \left( \left(\sigma - 1\right) \lx^2 - 2 K \right) \left(w_H^2 + 2\ly w_H w_K + \ly^2 w_K^2\right) \Big) \umbruch\\
    =&\, \frac{\sigma K^{\sigma-2}}{\left(w_H + \lx w_K\right)^2} \Big( \left(-\ly^2 - \lx^2 - 2K + \sigma\left(\ly^2 - 2K + \lx^2 \right) \right) w_H^2 \Big. \\
      &\, + \Big. \left( -2\lx\ly^2 - 2\lx^2\ly - 4K\ly + \sigma\left(2\lx\ly^2 - 2K\left(\lx+\ly\right) + 2\lx^2\ly\right)\right) w_H w_K \Big. \\
      &\, + \Big. \left( -K^2 - K^2 - 2K\ly^2 + \sigma\left(K^2 - 2K^2 + K^2\right)\right) w_K^2 \Big) \umbruch\\
    =&\, \frac{\sigma K^{\sigma-2}}{\left(w_H + \lx w_K\right)^2} \Big( \left(-\left(\lx+\ly\right)^2 + \sigma\left(\lx-\ly\right)^2\right) w_H^2 \Big. \\
      &\, \qquad \qquad \qquad \qquad - \Big. 2 \left(\lx + 3\ly \right) K w_H w_K - 2 \ly \left(\lx + \ly\right) K w_K^2 \Big).
  \end{align*}  

We calculate the gradient terms $G_K\left(\lx,\ly\right)$ for $F=K^\sigma$.
  \begin{align*}
      &\, G_K\left(\lx,\ly\right) \umbruch\\
    =&\, 2 \left(-\frac{\partial F}{\partial \lx} a_1 + \frac{\partial F}{\partial \ly} a_1^2 \right) 
         + \left( \frac{\partial^2 F}{\partial \lx^2} + 2 \frac{\partial^2 F}{\partial \lx \partial \ly} a_1 + \frac{\partial^2 F}{\partial \ly^2} a_1^2 \right) \ly 
         + 2 \frac{\frac{\partial F}{\partial \lx} - \frac{\partial F}{\partial \ly}}{\lx - \ly} \lx a_1^2 \\
      &\, \qquad \textbf{(\text{see}\,\eqref{K gradient ident})} \umbruch\\
    =&\, 2 \left(-\sigma K^{\sigma-1} \ly a_1 + \sigma K^{\sigma-1} \lx a_1^2 \right) \\
       &\, + \left( \sigma \left(\sigma-1\right) K^{\sigma-2} \ly^2 + 2\sigma^2 K^{\sigma-1} a_1 + \sigma\left(\sigma-1\right) K^{\sigma-2} \lx^2 a_1^2\right) \ly \\
       &\, + 2 \frac{\sigma K^{\sigma-1}\ly - \sigma K^{\sigma-1}\lx}{\lx-\ly} \lx a_1^2 \umbruch\\
    =&\, \sigma K^{\sigma-2} \Big( -2 K \ly a_1 + 2 K \lx a_1^2 + \left( \left(\sigma-1\right) \ly^2 + 2\sigma K a_1 + \left(\sigma-1\right) \lx^2 a_1^2\right) \ly \\
       &\, \qquad \qquad - 2 K \lx a_1^2 \Big) \umbruch\\
    =&\, \sigma K^{\sigma-2} \Big( \left(\sigma-1\right) \ly^3 + \left(-2 K \ly + 2\sigma K \ly \right) a_1 + \left(\sigma -1\right) \lx^2 \ly a_1^2 \Big) \umbruch\\
    =&\, \sigma\left(\sigma-1\right) K^{\sigma-2} \Big(\ly^3 + 2K \ly a_1 + \lx^2 \ly a_1^2 \Big) \umbruch\\
    =&\, \sigma \left(\sigma-1\right) K^{\sigma-2} \ly \Big(\ly^2 + 2K a_1 + \lx^2 a_1^2 \Big) \umbruch\\
    =&\, \sigma \left(\sigma-1\right) K^{\sigma-2} \ly \left(\ly + \lx a_1 \right)^2 \umbruch\\
    =&\, \sigma \left(\sigma-1\right) K^{\sigma-2} \ly \left(\ly - \lx \frac{w_H + \ly w_K}{w_H + \lx w_K} \right)^2 
             \qquad \textbf{(\text{see}\;\eqref{h111})} \umbruch\\
    =&\, \frac{\sigma K^{\sigma-2}}{\left(w_H + \lx w_K\right)^2} \left(\sigma-1\right) \ly \Big( \ly \left(w_H + \lx w_K \right) - \lx \left(w_H + \ly w_K\right)\Big)^2 \umbruch\\
    =&\, \frac{\sigma K^{\sigma-2}}{\left(w_H + \lx w_K\right)^2} \left(\sigma-1\right) \ly \Big( \ly w_H + K w_K - \lx w_H - K w_K \Big)^2 \umbruch\\
    =&\, \frac{\sigma K^{\sigma-2}}{\left(w_H + \lx w_K\right)^2} \left(\sigma-1\right) \left(\lx-\ly\right)^2 \ly w_H^2.
  \end{align*}  
  
We calculate the mixed terms \textbf{(compare \eqref{w gradient ident})} \\
$$-\frac{\partial F}{\partial \lx} \left(\left(1+a_1\right)^2w_{HH}+2\left(1+a_1\right)\left(\ly+\lx a_1\right)w_{HK} + \left(\ly+\lx a_1\right)^2w_{KK}\right)$$
for $F=K^\sigma$.
  \begin{align*}
      &\, -\frac{\partial F}{\partial \lx} \left(\left(1+a_1\right)^2w_{HH}+2\left(1+a_1\right)\left(\ly+\lx a_1\right)w_{HK} + \left(\ly+\lx a_1\right)^2w_{KK}\right) \umbruch\\
    =&\,-\sigma K^{\sigma-1}\ly \left(\left(1-\frac{w_H+\ly w_K}{w_H +\lx w_K}\right)^2w_{HH} \right. \\
       &\,+ \left. 2\left(1-\frac{w_H+\ly w_K}{w_H +\lx w_K}\right)\left(\ly-\frac{w_H+\ly w_K}{w_H +\lx w_K}\right)w_{HK} \right. \\
       &\,+ \left. \left(\ly-\lx \frac{w_H+\ly w_K}{w_H +\lx w_K}\right)^2w_{KK}\right) \umbruch\\
    =&\, -\frac{\sigma K^{\sigma-1} \ly}{\left(w_H + \lx w_K\right)^2} \Big( \left(w_H + \lx w_K - w_H - \ly w_K\right)^2 w_{HH} \\
       &\, \qquad + 2\left(w_H + \lx w_K - w_H - \ly w_K \right)\cdot \Big. \\
       &\, \qquad \qquad \Big. \cdot \left(\ly w_H + \lx \ly w_K - \lx w_H - \lx \ly w_K \right) w_{HK} \Big. \\
       &\, \qquad + \Big. \left(\ly w_H + \lx \ly w_K - \lx w_H - \lx \ly w_K \right)^2 w_H^2 w_{KK} \Big)
    \qquad \textbf{(\text{see}\;\eqref{h111})} \umbruch\\
    =&\, -\frac{\sigma K^{\sigma-1} \ly}{\left(w_H + \lx w_K\right)^2} \Big( \left(\lx-\ly\right)^2 w_K^2 w_{HH} - 2\left(\lx - \ly\right)^2 w_H w_K w_{HK} \Big. \\
       &\, \qquad \qquad \qquad \qquad \quad \Big. + \left(\lx-\ly\right)^2 w_H^2 w_{KK} \Big) \umbruch\\
    =&\, \frac{\sigma K^{\sigma-2}}{\left(w_H + \lx w_K\right)^2} \lx \left(\lx - \ly\right)^2 \ly^2 \Big(-w_K^2 w_{HH} + 2 w_H w_K w_{KK} - w_H^2 w_{KK} \Big).
  \end{align*}
 
Using $q^2 L\left(\frac{p}{q} \right) = q\;L\left(p\right) - p\;L\left(q\right)$ \eqref{pq evol} we calculate the gradient terms $G_{\frac{p}{q}}\left(\lx,\ly\right)$.
  \begin{align*}
      &\, q^2\,G_{\frac{p}{q}}\left(\lx,\ly\right) \umbruch\\
    =&\, G_H\left(\lx,\ly\right) r_H + G_K\left(\lx,\ly\right) r_K \umbruch\\
       &\, -\frac{\partial F}{\partial \lx} \left(\left(1+a_1\right)^2r_{HH}+2\left(1+a_1\right)\left(\ly+\lx a_1\right)r_{HK} + \left(\ly+\lx a_1\right)^2r_{KK}\right) \\
       &\, \qquad \textbf{(\text{see}\;\eqref{pq evol},\;\eqref{w evol ident})} \umbruch \\
    =&\, \frac{\sigma K^{\sigma-2}}{\left(r_H + \lx r_K\right)^2} \bigg(\left(-\left(\lx+\ly\right)^2 + \sigma \left(\lx - \ly\right)^2 \right) r_H^3 \bigg. \\
                                         &\, \qquad \qquad \qquad \quad \bigg. -2\left(\lx + 3\ly\right) K r_H^2 r_K - 2 \lx \left(\lx + \ly\right) \ly^2 r_H r_K^2 \bigg. \\
                                         &\, \qquad \qquad \qquad \quad \bigg. +\left(\sigma-1\right) \left(\lx-\ly\right)^2 \ly r_H^2 r_K \bigg. \\
                                         &\, \qquad \qquad \qquad \quad \bigg. -\lx \left(\lx-\ly\right)^2 \ly^2 r_K^2 r_{HH} \bigg. \\
                                         &\, \qquad \qquad \qquad \quad \bigg. +2\lx\left(\lx-\ly\right)^2\ly^2 r_H r_K r_{HK} \bigg. \\
                                         &\, \qquad \qquad \qquad \quad \bigg. -\lx\left(\lx-\ly\right)^2\ly^2r_H^2 r_{KK} \bigg)  
                                               \qquad \textbf{\text{(see \eqref{H gradient Ksigma}, \eqref{K gradient Ksigma}, \eqref{mixed terms Ksigma})}} \umbruch\\
     =&\, \frac{\sigma K^{\sigma-2}}{\left(r_H + \lx r_K\right)^2} 
                                                                                              \bigg(\left( \left(\sigma-1\right)\lx^2-2\left(\sigma+1\right)\lx \ly + \left(\sigma-1\right)\ly^2 \right) r_H^3 \bigg. \umbruch\\
                                         &\, \qquad \qquad \qquad \quad \bigg. +\left( \left(\sigma-3\right) \lx^2 
                                         - 2\left(\sigma+2\right) \lx \ly + \left(\sigma-1\right) \ly^2 \right) \ly r_H r_K^2 \bigg. \umbruch\\
                                         &\, \qquad \qquad \qquad \quad \bigg. -2 \lx \left(\lx+\ly\right) \ly^2 r_H r_K^2 \bigg. \umbruch\\
                                         &\, \qquad \qquad \qquad \quad \bigg. -\lx \left(\lx-\ly\right)^2 \ly^2 r_K^2 r_{HH} \bigg. \umbruch\\
                                         &\, \qquad \qquad \qquad \quad \bigg. +2\lx\left(\lx-\ly\right)^2\ly^2 r_H r_K r_{HK} \bigg. \umbruch\\
                                         &\, \qquad \qquad \qquad \quad \bigg. -\lx\left(\lx-\ly\right)^2\ly^2r_H^2 r_{KK} \bigg).  
  \end{align*}
  
Dividing by the nonnegative factor $\frac{\sigma K^{\sigma-2}}{q^2\left(r_H + \lx r_K\right)^2}$ we get the polynomial in two variables $G_r\left(\lx,\ly\right)$.
  \begin{align*}
    G_r\left(\lx,\ly\right) :=&\, \frac{q^2\left(r_H + \lx r_K\right)^2}{\sigma K^{\sigma-2}} G_{\frac{p}{q}}\left(\lx,\ly\right) \umbruch\\
                                       =&\, \left( \left(\sigma-1\right)\lx^2-2\left(\sigma+1\right)\lx \ly + \left(\sigma-1\right)\ly^2 \right) r_H^3  \\
                                          &\, + \left( \left(\sigma-3\right) \lx^2 - 2\left(\sigma+2\right) \lx \ly + \left(\sigma-1\right) \ly^2 \right) \ly r_H^2 r_K  \\
                                          &\, -2 \lx \left(\lx+\ly\right) \ly^2 r_H r_K^2  \\
                                          &\, -\lx \left(\lx-\ly\right)^2 \ly^2 r_K^2 r_{HH}  \\
                                          &\, +2\lx\left(\lx-\ly\right)^2\ly^2 r_H r_K r_{HK}  \\
                                          &\, -\lx\left(\lx-\ly\right)^2\ly^2r_H^2 r_{KK}.
  \end{align*}
  
Now we dehomogenize the previous polynomial setting $\lx=\rho,\,\ly=1$ and $\lx=1,\,\ly=\rho$, respectively. We get the two polynomials in one variable $G_1\left(\rho\right)$ and $G_2\left(\rho\right)$.
  \begin{align*}
    G_1\left(\rho\right) :=&\, G_r\left(\rho,1\right) \\
                                       =&\, \left( \left(\sigma-1\right)\rho^2-2\left(\sigma+1\right)\rho + \left(\sigma-1\right) \right) r_H^3  \\
                                          &\, + \left( \left(\sigma-3\right) \rho^2 - 2\left(\sigma+2\right) \rho + \left(\sigma-1\right) \right) r_H^2 r_K  \\
                                          &\, -2 \rho \left(\rho+1\right) r_H r_K^2  \\
                                          &\, -\rho \left(\rho-1\right)^2 r_K^2 r_{HH}  \\
                                          &\, +2\rho\left(\rho-1\right)^2 r_H r_K r_{HK}  \\
                                          &\, -\rho\left(\rho-1\right)^2 r_H^2 r_{KK}, \\
   \end{align*}
   \begin{align*}                                       
    G_2\left(\rho\right) :=&\, G_r\left(1,\rho\right) \\
                                       =&\, \left( \left(\sigma-1\right)\rho^2-2\left(\sigma+1\right) \rho + \left(\sigma-1\right) \right) r_H^3  \\
                                          &\, + \rho \left( \left(\sigma-1\right)\rho^2 - 2\left(\sigma+2\right) \rho + \left(\sigma-3\right) \right) r_H^2 r_K  \\
                                          &\, -2 \rho^2 \left(\rho+1\right) r_H r_K^2  \\
                                          &\, -\rho^2 \left(\rho-1\right)^2 r_K^2 r_{HH}  \\
                                          &\, +2\rho^2\left(\rho-1\right)^2 r_H r_K r_{HK}  \\
                                          &\, -\rho^2\left(\rho-1\right)^2 r_H^2 r_{KK}.
  \end{align*}
 
\end{proof}

\section{Dehomogenized polynomials, leading terms and nine cases}\label{sec dehomogenized polynomials}

\subsection{Dehomogenized polynomials, leading terms}

In the first part of Section \ref{sec dehomogenized polynomials}, we define homogeneous symmetric polynomials in the algebraic basis $\lbrace H,K \rbrace$.
We also state their first and second derivatives with respect to $H$ and $K$. 
Furthermore, we calculate their dehomogenized versions setting $\lx=\rho,\,\ly=1$. So we obtain several polynomials in one variable. 

Then we define an operator $\L$ that determines the leading terms of a given polynomial in one variable.
Now we present the leading terms of the above polynomials in one veriable. Here, we have to distinguish three distinct cases.

Due to the form of the $r$-terms this means that we have nine different cases to explore.
In the second part of Section \ref{sec dehomogenized polynomials} we determine the leading terms of the $r$-terms.
In each case we continue with the calculation of the leading terms of the polynomial constant terms $C\left(\rho\right)$ and 
the calculation of the leading terms of the polynomial gradient terms $G_1\left(\rho\right)$ and $G_2\left(\rho\right)$. 
All nine cases result in a contradiction. This concludes the proof of our main Theorem \ref{main theorem}.

\begin{lemma}\label{lem poly p}
  We define two homogeneous symmetric polynomials
  \begin{align}\label{poly p}
    p(H,K) :=&\,\sum_{i=0}^{\lfloor g/2\rfloor} c_{i+1}H^{g-2i}K^i, \\
    q(H,K) :=&\,\sum_{j=0}^{\lfloor h/2\rfloor} d_{j+1}H^{h-2j}K^j,
  \end{align}
  where $g$ is the degree of $p\left(H,K\right)$, and $h$ is the degree of $q\left(H,K\right)$, respectively.
  Furthermore, we have $\#\lbrace c_i \rbrace = \lfloor g/2 + 1 \rfloor$ and
  $\#\lbrace d_j \rbrace = \lfloor h/2 + 1 \rfloor$, respectively.
  
  We calculate the derivatives of the polynomial $p(H,K)$
  \begin{align*}
    \frac{\partial p}{\partial H}(H,K) :=&\,\sum_{i=0}^{\lfloor g/2\rfloor} c_{i+1}(g-2i)H^{g-2i-1}K^i, \\
    \frac{\partial p}{\partial K}(H,K) :=&\,\sum_{i=0}^{\lfloor g/2\rfloor} c_{i+1} i H^{g-2i}K^{i-1}, \\
    \frac{\partial^2 p}{\partial H^2}(H,K) :=&\,\sum_{i=0}^{\lfloor g/2\rfloor} c_{i+1}(g-2i)(g-2i-1)H^{g-2i-2}K^i, \\
    \frac{\partial^2 p}{\partial H \partial K}(H,K) :=&\,\sum_{i=0}^{\lfloor g/2\rfloor} c_{i+1} (g-2i) i H^{g-2i-1}K^{i-1}, \\
    \frac{\partial^2 p}{\partial K^2}(H,K) :=&\,\sum_{i=0}^{\lfloor g/2\rfloor} c_{i+1} i (i-1) (g-2i)H^{g-2i}K^{i-2}.
  \end{align*}
\end{lemma}

\begin{lemma}\label{lem poly}
  We calculate the dehomogenized version of the polynomial $p(H,K)$ and its derivatives from Lemma \ref{lem poly p} setting $\lx=1,\,\ly=\rho$
  \begin{align*}
        p(\rho) :=&\,\sum_{i=0}^{\lfloor g/2\rfloor} c_{i+1}(\rho+1)^{g-2i}\rho^i, \umbruch\\
         p_{H}(\rho) :=&\,\sum_{i=0}^{\lfloor g/2\rfloor} c_{i+1}(g-2i)(\rho+1)^{g-2i-1}\rho^i, \umbruch\\
         p_{K}(\rho) :=&\,\sum_{i=0}^{\lfloor g/2\rfloor} c_{i+1} i (\rho+1)^{g-2i}\rho^{i-1}, \umbruch\\
         p_{HH}(\rho) :=&\,\sum_{i=0}^{\lfloor g/2\rfloor} c_{i+1}(g-2i)(g-2i-1)(\rho+1)^{g-2i-2}\rho^i, \umbruch\\
         p_{HK}(\rho) :=&\,\sum_{i=0}^{\lfloor g/2\rfloor} c_{i+1} (g-2i) i (\rho+1)^{g-2i-1}\rho^{i-1}, \umbruch\\
         p_{KK}(\rho) :=&\,\sum_{i=0}^{\lfloor g/2\rfloor} c_{i+1} i (i-1) (g-2i)(\rho+1)^{g-2i}\rho^{i-2}.
  \end{align*}
\end{lemma}

\begin{remark}
  To determine the leading term of a polynomial $p \in \R [\rho]$ we write 
  \begin{align}\label{leading terms}
    \L\left(p\right) = c_g \rho^g,
  \end{align}
  if $p = c_g \rho^g + \sum_{i=0}^g c_i \rho^i$ for some $c_i \in \R$. 
  Note that $c_g = 0$ is possible.
  
  Furthermore, we set
  \begin{equation}
    \P_{\rho}\left(g\right) := \lbrace q \textit{ polynomial in } \rho :\textit{degree of } q \leq g \rbrace.
  \end{equation}
\end{remark}

\begin{lemma}\label{lem leading terms}
  We apply the operator $\L$ from Definition \ref{leading terms} to the polynomials in Lemma \ref{lem poly} in all three distinct cases. \\
  
\textbf{Case A.} $c_1> 0$
  \begin{align*}
    \L\left(p\right)=&\, c_1 \rho^g,\\
    \L\left(p_H\right)=&\,c_1 g \rho^{g-1},\\
    \L\left(p_K\right)=&\,c_2 \rho^{g-2},\\
    \L\left(p_{HH}\right)=&\,c_1 g\left(g-1\right)\rho^{g-2},\\
    \L\left(p_{HK}\right)=&\,c_2 \left(g-2\right)\rho^{g-3}, \\ 
    \L\left(p_{KK}\right)=&\,c_32\rho^{g-4}. 
  \end{align*}
  Terms with negative powers of $\rho$ do not occur for $g \geq 4$.
  If $g \leq 3$ the terms with negative powers of $\rho$ are $0$. \\

\textbf{Case B.} $c_1=0,\,c_2>0$
  \begin{align*}
    \L\left(p\right)=&\, c_2 \rho^{g-1},\\
    \L\left(p_H\right)=&\,c_2\left(g-2\right)\rho^{g-2},\\
    \L\left(p_K\right)=&\,c_2 \rho^{g-2},\\
    \L\left(p_{HH}\right)=&\,c_2 \left(g-2\right)\left(g-3\right)\rho^{g-3},\\
    \L\left(p_{HK}\right)=&\,c_2 \left(g-2\right)\rho^{g-3},\\
    \L\left(p_{KK}\right)=&\,c_32\rho^{g-4}. 
  \end{align*}
  Terms with negative powers of $\rho$ do not occur for $g \geq 4$.
  If $g \leq 3$ the terms with negative powers of $\rho$ are $0$. \\

\textbf{Case C.} $c_1=0,\ldots,c_{k-1}=0,\,c_k>0$ for all $k\geq 3$
  \begin{align*}
    \L\left(p\right)=&\, c_k \rho^{g-\left(k-1\right)},\\
    \L\left(p_H\right)=&\,c_k \left(g-2\left(k-1\right)\right) \rho^{g-k},\\
    \L\left(p_K\right)=&\,c_k \left(k-1\right)\rho^{g-k},\\
   \L\left(p_{HH}\right)=&\,c_k\left(g-2\left(k-1\right)\right)\left(g-2k+1\right) \rho^{g-\left(k+1\right)},\\
    \L\left(p_{HK}\right)=&\,c_k \left(g-2\left(k-1\right)\right)\left(k-1\right)\rho^{g-\left(k+1\right)},\\
    \L\left(p_{KK}\right)=&\,c_k\left(k-2\right)\left(k-1\right)\rho^{g-\left(k+1\right)}.
  \end{align*}
  Terms with negative powers of $\rho$ do not occur.
  Since $c_1=0,\ldots,c_{k-1}=0,\,c_k>0$ for all $k\geq 3$,
  we have $3 \leq k \leq \# \lbrace c_i \rbrace = \lfloor g/2 + 1\rfloor$.
  Thus, $2\left(k-1\right) \leq g$. Therefore, we get
  \begin{align*}
    &\, g- \left(k+1\right) 
    \geq 2\left(k-1\right) - \left(k+1\right)
    = k-3
    \geq 0.
  \end{align*}
  Furthermore, we have $g-k \geq 1$. We will use this implicitly in the second part of Section \ref{sec dehomogenized polynomials}.
\end{lemma}

\subsection{Nine cases}

\begin{remark}
  We recall from Lemma \ref{lem gradient terms} that the $r$-terms are defined as
 \begin{align*}
    &\, r_H:=q\;\frac{\partial p}{\partial H}-p\;\frac{\partial q}{\partial H},\;
    r_K:=q\;\frac{\partial p}{\partial K}-p\;\frac{\partial q}{\partial K}, \\
    &\, r_{HH}:=q\;\frac{\partial^2p}{\partial H^2}-p\;\frac{\partial^2q}{\partial H^2},\;
    r_{HK}:=q\;\frac{\partial^2p}{\partial H \partial K}-p\;\frac{\partial^2q}{\partial H \partial K},\;
    r_{KK}:=q\;\frac{\partial^2p}{\partial K^2}-p\;\frac{\partial^2q}{\partial K^2}.
  \end{align*}
  Therefore, we have to distinguish these nine cases in order to calculate the leading terms of the $r$-terms
  \begin{itemize}
    \item Case I: $c_1>0,\,d_1>0$,
    \item Case II: $c_1>0,\,d_2>0$,
    \item Case III: $c_1>0,\,d_l>0$,\\

    \item Case IV: $c_2>0,\,d_2>0$,
    \item Case V: $c_2>0,\,d_l>0$,
    \item Case VI: $c_k>0,\,d_l>0$,\\

    \item Case VII: $c_2>0,\,d_1>0$,
    \item Case VIII: $c_k>0,\,d_1>0$,
    \item Case IX: $c_k>0,\,d_2>0$,
  \end{itemize}
  for all $k,\,l \geq 3$.
\end{remark}

\begin{remark}
  We recall the constant terms $C\left(\rho\right)$ from Lemma \ref{lem constant terms}
  \begin{align}\label{C}
      C\left(\rho\right)=&\, \left(\left(1-\sigma\right)\rho^2+2\sigma\rho+\left(1-\sigma\right)\right)r_H + \rho\left(\rho+1\right)r_K,
  \end{align}
  the gradient terms $G_1\left(\rho\right)$ from Lemma \ref{lem gradient terms}
  \begin{align}\label{G1}
    \begin{split}
      G_1\left(\rho\right)= &\, \left( \left(\sigma-1\right)\rho^2-2\left(\sigma+1\right)\rho + \left(\sigma-1\right) \right) r_H^3  \\
                                        &\, + \left( \left(\sigma-3\right) \rho^2 - 2\left(\sigma+2\right) \rho + \left(\sigma-1\right) \right) r_H^2 r_K  \\
                                        &\, -2 \rho \left(\rho+1\right) r_H r_K^2  \\
                                        &\, -\rho \left(\rho-1\right)^2 r_K^2 r_{HH}  \\
                                        &\, +2\rho\left(\rho-1\right)^2 r_H r_K r_{HK}  \\
                                        &\, -\rho\left(\rho-1\right)^2 r_H^2 r_{KK},                             
    \end{split}
  \end{align}
  and the gradient terms $G_2\left(\rho\right)$ from Lemma \ref{lem gradient terms}
  \begin{align}\label{G2}
    \begin{split}
      G_2\left(\rho\right) =&\,\left( \left(\sigma-1\right)\rho^2-2\left(\sigma+1\right) \rho + \left(\sigma-1\right) \right) r_H^3  \\
                                        &\, + \rho \left( \left(\sigma-1\right)\rho^2 - 2\left(\sigma+2\right) \rho + \left(\sigma-3\right) \right) r_H^2 r_K  \\
                                        &\, -2 \rho^2 \left(\rho+1\right) r_H r_K^2  \\
                                        &\, -\rho^2 \left(\rho-1\right)^2 r_K^2 r_{HH}  \\
                                        &\, +2\rho^2\left(\rho-1\right)^2 r_H r_K r_{HK}  \\
                                        &\, -\rho^2\left(\rho-1\right)^2 r_H^2 r_{KK}.
    \end{split}
  \end{align}
\end{remark}

\subsection{Case I} $c_1>0,\,d_1>0$ \\
First we calculate the leading terms or the maximal order of the $r$-terms using Lemma \ref{lem leading terms}.
  \begin{align*}
    &\, \L\left(r_H\right) = c_1 d_1 \left(g-h\right) \rho^{g+h-1}, \\
    &\, \L\left(r_K\right) \in \P_{\rho}\left(g+h-2\right), \\
    &\, \L\left(r_{HH}\right) \in \P_{\rho}\left(g+h-2\right), \\
    &\, \L\left(r_{HK}\right) \in \P_{\rho}\left(g+h-3\right), \\
    &\, \L\left(r_{KK}\right) \in \P_{\rho}\left(g+h-4\right).
  \end{align*}
Now we calculate the leading terms of $G_1\left(\rho\right)$ using \eqref{G1}
  \begin{align*}
    \L\big(G_1\left(\rho\right)\big) =&\, c_1^3 d_1^3 \left(\sigma-1\right)\left(g-h\right)^3\rho^{3\left(g+h\right)-1}.
  \end{align*}
For \textit{maximum-principle functions} \eqref{def mpf} we have $g\geq 2$, $g-h > 0$ and $\L\big( G_1\left(\rho \right) \big) \leq 0$ for all $\rho \geq 0$. 
Since $\sigma-1 > 0$ \textit{Case I}, results in a contradiction.

\subsection{Case II} $c_1>0,\,d_2>0$ \\
First we calculate the leading terms or the maximal order of the $r$-terms using Lemma \ref{lem leading terms}.
  \begin{align*}
    &\, \L\left(r_H\right) = c_1 d_2 \left(g-h+2\right) \rho^{g+h-2}, \\
    &\, \L\left(r_K\right) = -c_1 d_2 \rho^{g+h-2}, \\
    &\, \L\left(r_{HH}\right) = c_1 d_2 \left(g-h+2\right)\left(g+h-3\right) \rho^{g+h-3}, \\
    &\, \L\left(r_{HK}\right) = -c_1 d_2 \left(h-2\right) \rho^{g+h-3}, \\
    &\, \L\left(r_{KK}\right) \in \P_{\rho}\left(g+h-4\right).
  \end{align*}
Now we calculate the leading terms of $G_1\left(\rho\right)$ using \eqref{G1}
  \begin{align*}
    &\, \L\big( G_1\left(\rho\right) \big) = -c_1^3 d_2^3 \left(g-h+1\right)\left(g-h+2\right)
    \left(\left(g-h\right)\left(1-\sigma\right)+1-2\sigma\right)\rho^{3\left(g+h\right)-4}. 
  \end{align*}
For \textit{maximum-principle functions} \eqref{def mpf} we have $g\geq 2$, $g-h > 0$ and $\L\big(G_1\left(\rho\right)\big) \leq 0$ for all $\rho\geq 0$. Since $\frac{2\sigma-1}{\sigma-2}>2$ for all $\sigma>1$, we get $g-h+\frac{2\sigma-1}{\sigma-1} > 0$ which is equivalent to $\left(g-h\right)\left(1-\sigma\right)+1-2\sigma < 0$. Therefore, \textit{Case II} results in a contradiction.

\subsection{Case III} $c_1>0,\,d_l>0$ for all $l\geq 3$ \\
First we calculate the leading terms of the $r$-terms using Lemma \ref{lem leading terms}.
  \begin{align*}
    &\, \L\left(r_H\right) = c_1 d_l \left(g-h+2\left(l-1\right)\right) \rho^{g+h-l}, \\
    &\, \L\left(r_K\right) = -c_1 d_l \left(l-1\right) \rho^{g+h-l}, \\
    &\, \L\left(r_{HH}\right) = c_1 d_l \left(g-h+2\left(l-1\right)\right)\left(g+h-2l+1\right)\rho^{g+h-\left(l+1\right)}, \\
    &\, \L\left(r_{HK}\right) = -c_1 d_l \left(h-2\left(l-1\right)\right)\left(l-1\right)\rho^{g+h-\left(l+1\right)}, \\
    &\, \L\left(r_{KK}\right) = -c_1 d_l \left(l-2\right)\left(l-1\right) \rho^{g+h-\left(l+1\right)}.
  \end{align*}
Now we calculate the leading terms of $G_1\left(\rho\right)$ using \eqref{G1}
  \begin{align*}
    &\, \L\big(G_1\left(\rho\right)\big) = -c_1^3 d_l^3 \left(g-h+\left(l-1\right)\right)\left(g-h+2\left(l-1\right)\right)\cdot \\
    &\, \qquad \qquad \qquad \cdot \left(\left(g-h\right)\left(1-\sigma\right)+\left(l-1\right)\left(1-2\sigma\right)\right) \rho^{3\left(g+h-l\right)+2}. 
  \end{align*}
For \textit{maximum-principle functions} \eqref{def mpf} we have $h-l\geq 1$, $g-h>0$ and $\L\big(G_1\left(\rho\right)\big) \leq 0$ for all $\rho\geq 0$. 
Since $\left(l-1\right)\frac{2\sigma-1}{\sigma-1} > 4$ for all $l\geq 3,\,\sigma>1$, we get $g-h+\left(l-1\right)\frac{2\sigma-1}{\sigma-1} > 0$
which is equivalent to $\left(g-h\right)\left(1-\sigma\right)+\left(l-1\right)\left(1-2\sigma\right) < 0$. Therefore, \textit{Case III} results in a contradiction.

\subsection{Case IV} $c_2>0,\,d_2>0$ \\
First we calculate the leading terms or the maximal order of the $r$-terms using Lemma \ref{lem leading terms}.
  \begin{align*}
    &\, \L\left(r_H\right) = c_2 d_2 \left(g-h\right) \rho^{g+h-3}, \\
    &\, \L\left(r_K\right) \in \P_{\rho}\left(g+h-4\right), \\
    &\, \L\left(r_{HH}\right) \in \P_{\rho}\left(g+h-4\right), \\
    &\, \L\left(r_{HK}\right) \in \P_{\rho}\left(g+h-4\right), \\
    &\, \L\left(r_{KK}\right) \in \P_{\rho}\left(g+h-5\right).
  \end{align*}
Now we calculate the leading terms of $G_1\left(\rho\right)$ using \eqref{G1}
  \begin{align*}
    \L\big(G_1\left(\rho\right)\big) =&\, c_2^3 d_2^3 \left(\sigma-1\right)\left(g-h\right)^3\rho^{3\left(g+h\right)-7}.
  \end{align*}
For \textit{maximum-principle functions} \eqref{def mpf} we have $g\geq 2$, $g-h>0$ and $\L\big(G_1\left(\rho\right)\big) \leq 0$ for all $\rho\geq 0$. 
Due to $d_1=0,\,d_2>0$ we have $h\geq 2$.
Since $\sigma-1>0$, \textit{Case IV} results in a contradiction.
 
\subsection{Case V} $c_2>0,\,d_l > 0$ for all $l\geq 3$ \\
First we calculate the leading terms of the $r$-terms using Lemma \ref{lem leading terms}.
  \begin{align*}
    &\, \L\left(r_H\right)= c_2d_l\left(g-h+2\left(l-2\right)\right)\rho^{g+h-\left(l+1\right)}, \\
    &\, \L\left(r_K\right)= -c_2d_l\left(l-2\right)\rho^{g+h-\left(l+1\right)},\\
    &\, \L\left(r_{HH}\right) = c_2d_l(g-h+2(l-2))(g+h-(2l+1))\rho^{g+h-\left(l+2\right)}, \\
    &\, \L\left(r_{HK}\right)= c_2d_l \left( \left(g-2\right) - \left(h-2\left(l-1\right)\right)\left(l-1\right) \right)\rho^{g+h-\left(l+2\right)}, \\
    &\, \L\left(r_{KK}\right)= -c_2d_l\left(l-2\right)\left(l-1\right)\rho^{g+h-\left(l+2\right)}.
  \end{align*}
Now we calculate the leading terms of $G_1\left(\rho\right)$ using \eqref{G1}
  \begin{align*}
    &\, \L\big(G_1\left(\rho\right)\big) = -c_2^3 d_l^3 (g-h+(l-2))(g-h+2(l-2))\cdot \\
    &\, \qquad \qquad \qquad \cdot((g-h)(1-\sigma)+(l-2)(1-2\sigma)) \rho^{3\left(g+h-l\right)-1}. 
  \end{align*}
For \textit{maximum-principle functions} \eqref{def mpf} we have $h-l \geq 1$, $g-h>0$ and $\L\big(G_1\left(\rho\right)\big) \leq 0$ for all $\rho\geq 0$. 
Since $\left(l-2\right)\frac{2\sigma-1}{\sigma-1} > 2$ for all $l\geq 3,\,\sigma>1$, we get $g-h+\left(l-2\right)\frac{2\sigma-1}{\sigma-1} > 0$
which is equivalent to $\left(g-h\right)\left(1-\sigma\right)+\left(l-2\right)\left(1-2\sigma\right) < 0$. Therefore, \textit{Case V} results in a contradiction.

\subsection{Case VI} $c_k>0,\,d_l>0$ for all $k,\,l\geq 3$ \\
First we calculate the leading terms of the $r$-terms using Lemma \ref{lem leading terms}.
  \begin{align*}
    &\, \L\left(r_H\right)= c_k d_l\left(\left(g-h\right)-2\left(k-l\right)\right)\rho^{g+h-\left(k+l-1\right)}, \\
    &\, \L\left(r_K\right)= c_kd_l\left(k-l\right)\rho^{g+h-\left(k+l-1\right)},\\
    &\, \L\left(r_{HH}\right)= c_kd_l(g+h+3-2(k+l))(g+h-2(k-l))\rho^{g+h-\left(k+l\right)}, \\
    &\, \L\left(r_{HK}\right)= c_kd_l\left(\left(g-2\left(k-1\right)\right)\left(k-1\right)-\left(h-2\left(l-1\right)\right)\left(l-1\right)\right)\rho^{g+h-\left(k+l\right)}, \\
    &\, \L\left(r_{KK}\right)= c_kd_l\left(\left(k-2\right)\left(k-1\right)-\left(l-2\right)\left(l-1\right)\right)\rho^{g+h-\left(k+l\right)}. 
\end{align*}
Now we calculate the leading terms of $C\left(\rho\right),\,G_1\left(\rho\right),\,G_2\left(\rho\right)$ using \eqref{C}, \eqref{G1}, \eqref{G2}
  \begin{align*}
    &\, \L\big(C(\rho)\big)= c_k d_l \left((g-h)(1-\sigma)+(l-k)(1-2\sigma)\right) \rho^{g+h-\left(k+l\right)+3},\\
    &\, \L\big(G_1(\rho)\big)= -c_k^3 d_l^3 (g-h+(l-k))(g-h+2(l-k))\cdot \\
    &\, \qquad \qquad \qquad \cdot((g-h)(1-\sigma)+(l-k)(1-2\sigma))\rho^{3\left(g+h-\left(k+l\right)\right)+5},\\
    &\, \L\big(G_2(\rho)\big)= -c_k^3 d_l^3 (l-k)(g-h+2(l-k))\cdot \\
    &\, \qquad \qquad \qquad \cdot((l-k)+(g-h+2(l-k))\sigma)\rho^{3\left(g+h-\left(k+l\right)\right)+6}. 
  \end{align*}
For \textit{maximum-principle functions} \eqref{def mpf} we have $g-k \geq 1$, $h-l \geq 1$, $g-h>0$ and $\L\big(C\left(\rho\right)\big) \leq 0$,
$\L\big(G_1\left(\rho\right)\big) \leq 0$, $\L\big(G_2\left(\rho\right)\big) \leq 0$ for all $\rho\geq 0$. \\
We assume $(g-h)(1-\sigma)+(l-k)(1-2\sigma) = 0$ which is equivalent to the identity $g-h = \left(k-l\right)\frac{2\sigma-1}{\sigma-1}$. 
Since $\frac{2\sigma-1}{\sigma-1} > 2$, we get $k-l > 0$. Using this identity we get
  \begin{align*}
    \L\big(G_2(\rho)\big)=&\, c_k^3 d_l^3 (k-l)^3\left(\frac{1}{\sigma-1}\right)^2\rho^{3\left(g+h-\left(k+l\right)\right)+6}
  \end{align*}
which results in a contradiction. So we have $(g-h)(1-\sigma)+(l-k)(1-2\sigma) < 0$ which is equivalent to $g-h + \left(l-k\right)\frac{2\sigma-1}{\sigma-1} > 0$.
Furthermore, we assume $k-l > 0$ which implies
  \begin{align*}
    &\, g-h>(k-l)\frac{2\sigma-1}{\sigma-1}>2(k-l)>k-l \text{ and }\\
    &\, g-h+(l-k)>g-h+2(l-k)>0.
  \end{align*}
Thus, the condition $\L\big(G_1(\rho)\big)\leq 0$ for all $\rho\geq 0$ results in a contradiction. 
For $l-k\geq0$ the same condition also results in a contradiction. Therefore, \textit{Case VI} results in a contradiction.

\subsection{Case VII} $c_2>0,\, d_1>0$ \\
First we calculate the leading terms of the $r$-terms using Lemma \ref{lem leading terms}.
  \begin{align*}
    &\, \L\left(r_H\right) = c_2 d_1 \left(g-h -2\right) \rho^{g+h-2}, \\
    &\, \L\left(r_K\right) = c_2 d_1 \rho^{g+h-2}, \\
    &\, \L\left(r_{HH}\right) = c_2 d_1 \left(g-h-2\right)\left(g+h-3\right) \rho^{g+h-3}, \\
    &\, \L\left(r_{HK}\right) = c_2 d_1 \left(g-2\right) \rho^{g+h-3}, \\
    &\, \L\left(r_{KK}\right) \in \P_{\rho}\left(g+h-4\right).
  \end{align*}
Now we calculate the leading terms of $C\left(\rho\right),\,G_1\left(\rho\right),\,G_2\left(\rho\right)$ using \eqref{C}, \eqref{G1}, \eqref{G2}
  \begin{align*}
    &\, \L\big(C\left(\rho\right)\big) = c_2 d_1 \left(\left(g-h\right)\left(1-\sigma\right) - 1 + 2\sigma\right)\rho^{g+h}, \\
    &\, \L\big(G_1\left(\rho\right)\big) = -c_2^3 d_1^3 \left(g-h-2\right)\left(g-h-1\right)
                                                                                      \left(\left(g-h\right)\left(1-\sigma\right)-1+2\sigma\right)\rho^{3\left(g+h\right)-4}, \\
    &\, \L\big(G_2\left(\rho\right)\big) = c_2^3 d_1^3 \left(-\left(g-h-2\right) - \left(g-h-2\right)^2\sigma\right)\rho^{3\left(g+h\right)-3}. 
  \end{align*}
For \textit{maximum-principle functions} \eqref{def mpf} we have $g\geq 2$, $g-h>0$ and $\L\big(C\left(\rho\right)\big) \leq 0$,
$\L\big(G_1\left(\rho\right)\big) \leq 0$, $\L\big(G_2\left(\rho\right)\big) \leq 0$ for all $\rho\geq 0$. \\
We assume $\left(g-h\right)\left(1-\sigma\right) - 1 + 2\sigma = 0$ which is equivalent to the identity $g-h = \frac{2\sigma-1}{\sigma-1}$. Using this identity we get
  \begin{align*}
    &\, \L\big(G_2\left(\rho\right)\big) = c_2^3 d_1^3 \left(\frac{1}{\sigma-1}\right)^2\rho^{3\left(g+h\right)-3}
  \end{align*}
which results in a contradiction. So we have $\left(g-h\right)\left(1-\sigma\right) - 1 + 2\sigma < 0$ which is equivalent to $g-h > \frac{2\sigma-1}{\sigma-1}$. This implies
  \begin{align*}
    &\, g-h > \frac{2\sigma-1}{\sigma-1} > 2 > 1 \text{ and } \\
    &\, g-h-1 > g-h-2 > 0.
  \end{align*}
Thus, the condition $\L\big(G_1\left(\rho\right)\big) \leq 0$ for all $\rho\geq 0$ results in a contradiction. Therefore, \textit{Case VII} results in a contradiction.
  
\subsection{Case VIII} $c_k>0,\, d_1>0$ for all $k\geq 3$ \\
First we calculate the leading terms of the $r$-terms using Lemma \ref{lem leading terms}.
  \begin{align*}
    &\, \L\left(r_H\right) = c_k d_1 \left(g-h -2\left(k-1\right)\right) \rho^{g+h-k}, \\
    &\, \L\left(r_K\right) = c_k d_1 \left(k-1\right) \rho^{g+h-k}, \\
    &\, \L\left(r_{HH}\right) = c_k d_1 \left(g-h-2\left(k-1\right)\right)\left(g+h-2k+1\right) \rho^{g+h-\left(k+1\right)}, \\
    &\, \L\left(r_{HK}\right) = c_k d_1 \left(g-2\left(k-1\right)\right)\left(k-1\right) \rho^{g+h-\left(k+1\right)}, \\
    &\, \L\left(r_{KK}\right) = c_k d_1 \left(k-2\right)\left(k-1\right) \rho^{g+h-\left(k+1\right)}.
  \end{align*}
Now we calculate the leading terms of $C\left(\rho\right),\,G_1\left(\rho\right),\,G_2\left(\rho\right)$ using \eqref{C}, \eqref{G1}, \eqref{G2}
  \begin{align*}
    &\, \L\big(C\left(\rho\right)\big) = c_k d_1 \left(\left(g-h\right)\left(1-\sigma\right) + \left(k-1\right)\left(-1+2\sigma\right)\right)\rho^{g+h-k+2}, \\
    &\, \L\big(G_1\left(\rho\right)\big) = -c_k^3 d_1^3 \left(g-h-2\left(k-1\right)\right)\left(g-h-\left(k-1\right)\right) \cdot \\
    &\, \qquad \qquad \qquad \qquad \cdot \left(\left(g-h\right)\left(1-\sigma\right)+\left(k-1\right)\left(-1+2\sigma\right)\right)\rho^{3\left(g+h-k\right)+2}, \\
    &\, \L\big(G_2\left(\rho\right)\big) = c_k^3 d_1^3 \left(g-h-2\left(k-1\right)\right)\left(k-1\right)\cdot \\
    &\, \qquad \qquad \qquad \qquad \cdot \left(-\left(k-1\right)+\left(g-h-2\left(k-1\right)\right)\sigma\right)\rho^{3\left(g+h-k\right)+3}. 
  \end{align*}
For \textit{maximum-principle functions} \eqref{def mpf} we have $g-k\geq 1$, $g-h>0$ and $\L\big(C\left(\rho\right)\big) \leq 0$,
$\L\big(G_1\left(\rho\right)\big) \leq 0$, $\L\big(G_2\left(\rho\right)\big) \leq 0$ for all $\rho\geq 0$. \\
We assume $\left(g-h\right)\left(1-\sigma\right) + \left(k-1\right)\left(-1+2\sigma\right) = 0$ which is equivalent to the identity $g-h = \left(k-1\right)\frac{2\sigma-1}{\sigma-1}$. Using this identity we get
  \begin{align*}
    &\, \L\big(G_2\left(\rho\right)\big) = c_k^3 d_1^3 \left(k-1\right)^3 \left(\frac{1}{\sigma-1}\right)^2\rho^{3\left(g+h-k\right)+3}
  \end{align*}
which results in a contradiction. So we have $\left(g-h\right)\left(1-\sigma\right) + \left(k-1\right)\left(-1+2\sigma\right) < 0$ which is equivalent to $g-h > \left(k-1\right)\frac{2\sigma-1}{\sigma-1}$. This implies
  \begin{align*}
    &\, g-h > \left(k-1\right) \frac{2\sigma-1}{\sigma-1} > 2\left(k-1\right) > k-1 \text{ and } \\
    &\, g-h-\left(k-1\right) > g-h-2\left(k-1\right) > 0.
  \end{align*}
Thus, the condition $\L\big(G_1\left(\rho\right)\big) \leq 0$ for all $\rho\geq 0$ results in a contradiction. Therefore, \textit{Case VIII} results in a contradiction.

\subsection{Case IX} $c_k>0,\, d_2>0$ for all $k\geq 3$ \\
First we calculate the leading terms of the $r$-terms using Lemma \ref{lem leading terms}.
  \begin{align*}
    &\, \L\left(r_H\right) = c_k d_2 \left(g-h -2\left(k-2\right)\right) \rho^{g+h-\left(k+1\right)}, \\
    &\, \L\left(r_K\right) = c_k d_2 \left(k-2\right) \rho^{g+h-\left(k+1\right)}, \\
    &\, \L\left(r_{HH}\right) = c_k d_2 \left(g-h-2\left(k-2\right)\right)\left(g+h-\left(2k+1\right)\right) \rho^{g+h-\left(k+2\right)}, \\
    &\, \L\left(r_{HK}\right) = c_k d_2 \left(\left(g-2\left(k-1\right)\right)\left(k-1\right) - \left(h-2\right)\right) \rho^{g+h-\left(k+2\right)}, \\
    &\, \L\left(r_{KK}\right) = c_k d_2 \left(k-2\right)\left(k-1\right) \rho^{g+h-\left(k+2\right)}.
  \end{align*}
Now we calculate the leading terms of $C\left(\rho\right),\,G_1\left(\rho\right),\,G_2\left(\rho\right)$ using \eqref{C}, \eqref{G1}, \eqref{G2}
  \begin{align*}
    &\, \L\big(C\left(\rho\right)\big) = c_k d_1 \left(\left(g-h\right)\left(1-\sigma\right) + \left(k-2\right)\left(-1+2\sigma\right)\right)\rho^{g+h-\left(k-1\right)}, \\
    &\, \L\big(G_1\left(\rho\right)\big) = -c_k^3 d_2^3 \left(g-h-2\left(k-2\right)\right)\left(g-h-\left(k-2\right)\right) \cdot \\
    &\, \qquad \qquad \qquad \qquad \cdot \left(\left(g-h\right)\left(1-\sigma\right)+\left(k-2\right)\left(-1+2\sigma\right)\right)\rho^{3\left(g+h-k\right)-1}, \\
    &\, \L\big(G_2\left(\rho\right)\big) = c_k^3 d_1^3 \left(g-h-2\left(k-1\right)\right)\left(k-1\right)\cdot \\
    &\, \qquad \qquad \qquad \qquad \cdot \left(-\left(k-1\right)+\left(g-h-2\left(k-1\right)\right)\sigma\right)\rho^{3\left(g+h-k\right)}. 
  \end{align*}
For \textit{maximum-principle functions} \eqref{def mpf} we have $g-k\geq 1$, $g-h>0$ and $\L\big(C\left(\rho\right)\big) \leq 0$,
$\L\big(G_1\left(\rho\right)\big) \leq 0$, $\L\big(G_2\left(\rho\right)\big) \leq 0$ for all $\rho\geq 0$. \\
We assume $\left(g-h\right)\left(1-\sigma\right) + \left(k-2\right)\left(-1+2\sigma\right) = 0$ which is equivalent to the identity $g-h = \left(k-2\right)\frac{2\sigma-1}{\sigma-1}$. Using this identity we get
  \begin{align*}
    &\, \L\big(G_2\left(\rho\right)\big) = c_k^3 d_2^3 \left(k-2\right)^3 \left(\frac{1}{\sigma-1}\right)^2\rho^{3\left(g+h-k\right)}
  \end{align*}
which results in a contradiction. So we have $\left(g-h\right)\left(1-\sigma\right) + \left(k-2\right)\left(-1+2\sigma\right) < 0$ which is equivalent to $g-h > \left(k-2\right)\frac{2\sigma-1}{\sigma-1}$. This implies
  \begin{align*}
    &\, g-h > \left(k-2\right) \frac{2\sigma-1}{\sigma-1} > 2\left(k-2\right) > k-2 \text{ and } \\
    &\, g-h-\left(k-2\right) > g-h-2\left(k-2\right) > 0.
  \end{align*}
Thus, the condition $\L\big(G_1\left(\rho\right)\big) \leq 0$ for all $\rho\geq 0$ results in a contradiction. Therefore, \textit{Case IX} results in a contradiction.

\bibliographystyle{amsplain} 
\def\weg#1{} \def\unterstrich{\underline{\rule{1ex}{0ex}}} \def\cprime{$'$}
  \def\cprime{$'$} \def\cprime{$'$} \def\cprime{$'$}
\providecommand{\bysame}{\leavevmode\hbox to3em{\hrulefill}\thinspace}
\providecommand{\MR}{\relax\ifhmode\unskip\space\fi MR }
\providecommand{\MRhref}[2]{%
  \href{http://www.ams.org/mathscinet-getitem?mr=#1}{#2}
}
\providecommand{\href}[2]{#2}

\end{document}